\documentclass[a4paper,12pt,twoside]{article}

\usepackage[margin=3cm]{geometry}
\usepackage{amssymb}
\usepackage{amsmath}
\usepackage{amsthm}
\usepackage{amscd}
\usepackage{amstext}
\numberwithin{equation}{section}
\usepackage{enumerate}
\usepackage{amssymb}
\usepackage{amsmath}
\usepackage{amsfonts}
\usepackage{amsthm}
\usepackage[latin1]{inputenc}
\usepackage[T1]{fontenc}
\usepackage{enumerate}
\usepackage[american]{babel}
\usepackage{bbm,bm}
\usepackage[mathscr]{euscript}
\usepackage{stmaryrd}

\usepackage{mathtools}
\mathtoolsset{centercolon}

\usepackage[draft,danish]{fixme}
\usepackage{mathtools}
\usepackage{url}

\newcommand{\mathset}[1]{\mathbbm{#1}}

\newcommand{\R}{\mathset{R}}

\newcommand{\mcalF}{\mathcal{F}}
\newcommand{\mcalM}{\mathcal{M}}

\newcommand{\Ds}{{}^s\!}
\newcommand{\Dt}{{}^t\!}

\newcommand{\Dsu}{{}^s}

\newcommand{\ca}{c\`adl\`ag}

\newcommand{\ilaw}{\mathrel{\smash{\stackrel{\scriptscriptstyle\mathcal{D}}{=}}}}
\newcommand{\Peq}{\mathrel{\smash{\stackrel{\scriptscriptstyle{P}}{=}}}}

\newcommand{\inl}{\mathrel{\smash{\stackrel{\scriptscriptstyle\mathrm{in}}{=}}}}
\newcommand{\dd}{\mathrm{d}}

\newcommand{\gen}{\textup{g}}

\newcommand{\inbull}{\mathrel{\smash{\stackrel{\scriptscriptstyle\mathrm{in}}{\bullet}}}}

\newcommand\bdot{\bm{.}}

\theoremstyle{plain}
\newtheorem{lemma}{Lemma}[section]
\newtheorem{proposition}[lemma]{Proposition}
\newtheorem{theorem}[lemma]{Theorem}
\newtheorem{corollary}[lemma]{Corollary}
\theoremstyle{definition}

\theoremstyle{definition}
\newtheorem{definition}[lemma]{Definition}
\theoremstyle{remark}
\newtheorem{remark}[lemma]{Remark}
\newtheorem{example}[lemma]{Example}

\title{Martingale-type processes indexed by the real line}
  \author{Andreas Basse-O'Connor$^*$\ \    Svend-Erik Graversen$^*$\ \  Jan Pedersen\thanks{Department of Mathematical Sciences,
 University of Aarhus,\newline  Ny Munkegade, DK-8000 \AA rhus
   C, Denmark. E-mails: \{basse matseg jan\}@imf.au.dk} }
\date{}
\begin{document}\maketitle

\begin{abstract}
Some  classes  of increment  martingales,  and the
corresponding  localized
classes, are studied. An increment martingale is indexed by 
 $\R$ and its  increment processes are martingales. We
focus primarily on the behavior   as
time goes to $-\infty$ in relation to  the quadratic
variation or the predictable quadratic variation, and we relate the
limiting behaviour to the martingale property.  Finally,
integration with respect to an increment martingale   is studied. 

\medskip 


\medskip 


\end{abstract}

\section{Introduction}
Stationary processes are widely used in many areas, and  
the key example   is a moving average,  that is,  a process $X$ of the form 
\begin{equation} \label{iip}
X_t=\int_{-\infty}^t \psi(t-s)\,\mathrm{d}M_s,\qquad t\in\R, 
\end{equation} 
where $M=(M_t)_{t\in\R}$ is a  process with stationary increments and
$\psi: [0,\infty) \to \R$ is deterministic.  
A particular  example is a stationary Ornstein-Uhlenbeck process which
corresponds to the case  
$\psi(t)=e^{-\lambda t}$  and $M$ is a Brownian motion
indexed by $\R$. See  Doob (1990)   for second order properties of moving
averages and   Barndorff-Nielsen and Schmiegel (2008)  for their 
applications  in turbulence. Also note that \eqref{iip} can be
generalised in many directions. For example, if instead of integrating 
from $-\infty$ to $t$ we integrate over $\R$ and replace $\psi(t-s)$
by, say,  $\phi(t-s)- \phi(-s)$, where $\phi:  \R\to \R$ is
deterministic,    we would also be able to model
processes with stationary increments. In particular, in this  setting 
the  fractional Brownian with Hurst parameter
$H\in (0,1)$ corresponds   to  $\phi(t) =
t^{H-1/2}1_{\R_+}(t)$; see  Samorodnitsky and Taqqu (1994, Section
7.2). 

Integration with respect to a local martingale indexed by $\R_+$ is
well-developed and in this case one can even allow the integrand to be
random.   However,  when trying
to define a stochastic integral from $-\infty$ as in \eqref{iip}
with  random integrands,    the  class of local martingales indexed by  $\R$
does not  provide  the right framework for $M=(M_t)_{t\in \R}$; indeed, in
simple cases,   such as  when $M$   is a Brownian motion, $M$ is not a
martingale in any filtration.  Rather, it seems  better to think
of $M$ as a process for which the   increment    $(M_{t+s}- M_s)_{t\geq
  0}$ is a  martingale  for all $s\in \R$. 
It is natural to call such a
process an  \textit{increment martingale.} Another interesting
example within this framework is a diffusion on natural scale  started in
$\infty$ (cf.\ Example \ref{ivba}); indeed, if  $\infty$ is an entrance
boundary  then all  increments are local martingales but the diffusion
itself is not. Thus,  the class of increment (local) martingales indexed
by $\R$ is
strictly larger than the class of (local) martingales indexed by $\R$
and it contains several interesting examples. 
 We refer to
Subsection \ref{dtr} for a discussion of the  relations to other kinds
of  martingale-type processes indexed by  $\R$.

In the present paper we introduce and study basic properties of some 
classes  of increment martingales  $M=(M_t)_{t\in \R}$ and the corresponding  localized
classes. Some of the problems studied are the following.  Necessary and sufficient conditions for $M$ to be a
local martingale up to addition of a random variable will be given
when $M$ is either an increment  martingale  or an  increment square
integrable martingale. In addition,  we give various 
necessary and sufficient conditions for $M_{-\infty}= \lim_{t\to
  -\infty} M_t$ to exist  $P$-a.s.\ and $M-M_{-\infty}$ to be a local
martingale expressed  in terms of either the  predictable quadratic
variation 
$\langle M\rangle$ or the  quadratic variation $[M]$ for $M$, where the
latter two quantities will be defined below for increment
martingales. These conditions rely on a convenient decomposition of
increment martingales, and  are particularly simple when $M$ is
continuous.  We define two
kinds of integrals with respect to  $M$; the first of these  is an
increment integral $\phi\inbull M$, which we can think of as process
satisfying  $\phi\inbull M_t -\phi\inbull M_s
=\int_{(s,t]}\phi_u\, \dd M_u$; i.e.\ increments in  $\phi\inbull M$
correspond to integrals over finite intervals. The second integral,
$\phi\bullet M$, is a  usual stochastic integral with respect to $M$
which we can think of as an integral from $-\infty$. The integral
$\phi\bullet M$ exists if and only if the increment integral
$\phi\inbull M$ has an a.s.\  limit, $\phi\inbull M_{-\infty}$,  at
$-\infty$ and   
$\phi\inbull M - \phi\inbull M_{-\infty}$ is a local
martingale. Thus,  $\phi\inbull M_{-\infty}$ may exists without
$\phi \bullet M$ being defined and in this case we may think of
$\phi\inbull M_{-\infty}$ as an improper integral.  In special cases we give necessary
and sufficient conditions for  $\phi\inbull M_{-\infty}$ to exist.

The present paper relies only on standard
martingale results and martingale integration  as  developed in many
textbooks, see  e.g.\ 
Jacod and Shiryaev (2003) and Jacod (1979). While we focus primarily
on the behaviour at $-\infty$, it is also of interest to consider the
behaviour at $\infty$; we refer to Cherny and Shiryaev (2005), and
references therein,  for a
study of this case for semimartingales,  and to Sato (2006),  and
references therein,   for a
study of  improper integrals with respect to   L\'evy processes when the
integrand is deterministic. Finally, we note that having studied increment
martingales, it is natural to  introduce and study a concept called
\textit{increment semimartingales}; this will be  included in a
forthcoming paper by the authors; see  Basse-O'Connor et al.\ (2010). 

\subsection{Relations to other martingale-type processes} \label{dtr}
Let us briefly discuss how to define processes with some kind of 
martingale structure when  processes are indexed by   $\R$.
 There are at least three natural definitions:
 \begin{enumerate}[(i)]
 \item  \label{def_a}$E[M_t|\mcalF_s^M]=M_s$ for all
 $s\leq t$, where $\mcalF^M_s=\sigma(M_u:u\in (-\infty,s])$.
  \item \label{def_b} $E[M_t-M_u|\mcalF_{v,s}^{\mathcal{I}M}]=M_s-M_v$ for all
    $u\leq v\leq s\leq  t$, where
 $\mcalF^{\mathcal{I}M}_{v,s}=\sigma(M_t-M_u:v\leq u\leq t\leq s)$.
 \item \label{def_c} $E[M_t-M_s|\mcalF^{\mathcal{I}M}_s]=0$ for all $s\leq t$, where
 $\mcalF^{\mathcal{I}M}_s=\sigma(M_t-M_u:u\leq t\leq s)$.
 \end{enumerate}
(The first definition is the  usual martingale definition  and the third
one  corresponds to  increment martingales in the filtration
$(\mathcal{F}^{\mathcal{I}M}_t)_{t\in \R}$). 
 Both   \eqref{def_a} and \eqref{def_c}
 generalise the usual notion of martingales  indexed by 
 $\R_+$,  in
 the sense that if  $(M_t)_{t\in \R}$ is  a process with $M_t=0$ for $t\in
 (-\infty,0]$, then  $(M_t)_{t\geq 0}$ is a martingale (in the usually
 sense)  if and only if $(M_t)_{t\in\R}$ is a martingale in the sense of
 \eqref{def_a}, or equivalently in the sense of \eqref{def_c}. 
 Definition \eqref{def_b} does not generalise martingales indexed by  
 $\R_+$ in this manner. Note moreover that  a  centered  L\'evy process
indexed by    $\R$ (cf.\ Example \ref{levdef})  is a
 martingale in the sense of \eqref{def_b} and \eqref{def_c} but not in
 the sense of \eqref{def_a}. Thus, \eqref{def_c} is the only one of the
 above definitions
 which generalise the usual notion of martingales on $\R_+$
 and is  general enough to allow centered L\'evy processes to be
 martingales. Note also that both
 \eqref{def_a} and \eqref{def_b} imply \eqref{def_c}.

The general theory of martingales indexed by partially  ordered sets (for
 short, posets) does not seem to give us much  insight about increment
 martingales since the  research in this field mainly has  a  different focus;
 indeed,  one of the main problems has been to study martingales
 $M=(M_t)_{t\in I}$ in the case where $I=[0,1]^2$; see e.g.\ Cairoli
 and Walsh (1975,1977).   However,
  below we  recall some of the basic definitions and relate them to the above
 \eqref{def_a}--\eqref{def_c}.

Consider  a poset $(I,\leq)$ and a filtration
$\mcalF=(\mcalF_t)_{t\in I}$,  that
 is, for all  $s,t\in I$ with $s\leq t$ we have that
 $\mcalF_s\subseteq \mcalF_t$.   Then, $(M_t)_{t\in I}$ is called a martingale
 with respect to $\leq$ and $\mcalF$,  if for all $s,t\in I$ with
 $s\leq t$ we have  that
 $E[M_t|\mcalF_s]=M_s$.
 Let $M=(M_t)_{t\in \R}$ denote a stochastic process.
 Then, definition \eqref{def_a}  corresponds to
 $I=\R$ with the usually order. To cover \eqref{def_b} and
 \eqref{def_c} let  $I=\{(a_1,a_2]:a_1,a_2,\in\R,\
 a_1<a_2\}$, and for $A=(a_1,a_2]\in I$ let  $M_A=M_{a_2}-M_{a_1}$,
 $\mcalF_A^M=\sigma(M_B:B\in I,\ B\subseteq A)$.
 Furthermore, for all  $A=(a_1,a_2],B=(b_1,b_2]\in I$ we will
 write $A\leq_2\! B$ if
 $A\subseteq B$, and $A\leq_3\! B$  if
 $a_1=b_1$ and $a_2\leq b_2$. Clearly, $\leq_2$ and $\leq_3$
 are two partial orders on $I$. Moreover, it is easily seen
 that  $(M_t)_{t\in \R}$ satisfies  \eqref{def_b}/\eqref{def_c} if and only
 if $(M_A)_{A\in  I}$  is a martingale with respect to
 $\leq_2$/$\leq_3$ and $\mcalF^M$. Recall that a poset $(I,\leq)$ is called directed  if
 for all $s,t\in I$ there exists an element $u\in I$ such that $s\leq
 u$ and $t\leq u$. Note that $(I,\leq_2)$ is directed, but
 $(I,\leq_3)$ is not; and in particular $(I,\leq_3)$ is not a lattice.
 We refer to Kurtz (1980)  for some nice considerations about
 martingales indexed by directed posets.

\section{Preliminaries}\label{sect2}

Let $(\Omega,\mcalF, P)$ denote a complete
probability space on which all random variables appearing in the
following  are defined. 
Let $\mcalF_{\bdot}= (\mcalF_t)_{t\in \R}$ denote a filtration in
$\mcalF$, i.e.\  a right-continuous  increasing family of
sub $\sigma$-algebras in $\mcalF$ satisfying  $\mathcal{N}\subseteq \mcalF_t$ for all
$t$, where $\mathcal{N}$ is the collection of all $P$-null sets. Set  $\mcalF_{-\infty}:= \cap_{t\in \R} \mcalF_t$   and
$\mcalF_\infty:=  \cup_{t\in \R} \mcalF_t$.   
 The notation $\ilaw $ will be used to denote identity in
distribution. Similarly, $\Peq$ will denote equality up to
$P$-indistinguishability of stochastic processes. 
When $X=(X_t)_{t\in
  \R}$ is a real-valued    stochastic process we say that $\lim_{s \to
  -\infty}X_s $ exists $P$-a.s.\ if $X_s$ converges almost surely as $s\to
-\infty$,  to a finite limit.  

\begin{definition} 
 A  \textit{stopping time with respect to $\mcalF_{\bdot}$}  is a
 mapping   $\sigma:\Omega \to (-\infty, \infty]$ satisfying 
  $\{\sigma\leq t\}\in \mcalF_t$ for all $t\in \R$. (When there is no
  risk of confusion, we  often omit  terms like  "with respect
to  $\mcalF_{\bdot}$".) A \textit{localizing
  sequence}  $(\sigma_n)_{n\geq 1}$ is a sequence of stopping times
   satisfying $\sigma_1(\omega)\leq \sigma_2(\omega)\leq \cdots $ for
  all $\omega$, and $\sigma_n \to \infty$   $P$-a.s. 

  Let $\mathcal{P}(\mcalF_{\bdot})$ denote the
 \textit{predictable $\sigma$-algebra} on $\R\times \Omega$. 
  That is,  the $\sigma$-algebra generated by the set of \textit{simple
  predictable sets}, where  a subset of $\R\times \Omega$ is said
  to be simple predictable if it is of the form  $B\times C$
  where, for some $t\in \R$,   $C$ is in $\mcalF_t$  and $B$ is a bounded
  Borel set in $]t, \infty[$.  Note  that the set of simple
  predictable sets is closed under finite intersections.
\end{definition}

 Any left-continuous and adapted
 process is predictable. Moreover, the set of predictable processes is
 stable under stopping in the sense that whenever
 $\alpha=(\alpha_t)_{t\in \R}$ is
 predictable and $\sigma$ is a stopping time, the stopped process
 $\alpha^\sigma:= (\alpha_{t\wedge  \sigma})_{t\in \R}$ is also  predictable.

By an \textit{increasing process} we mean a   process $V=(V_t)_{t\in
  \R}$ (not necessarily  adapted)  for which   
$t\mapsto V_t(\omega)$ is nondecreasing for all $\omega\in
\Omega$. Similarly, a  process 
$V$   is said to be \ca\ if $t\mapsto V_t(\omega)$ is
right-continuous and has left limits in $\R$ for all $\omega\in \Omega$.   

In what follows  increments of  processes  play an important   role.  
Whenever $X=(X_t)_{t\in \R}$ is a process and $s,t\in \R$ define the
\textit{increment of  $X$ over the interval $(s,t]$}, to be denoted
$\Ds X_t$,  as 
\begin{equation}
\Ds X_t:= X_t - X_{t\wedge s} = \begin{cases}
0 &\mbox{if }t\leq s\\
X_t - X_s&\mbox{if } t\geq s.
\end{cases}
\end{equation} 
Set furthermore  $\Ds X= (\Ds X_t)_{t\in \R}$. 
Note  that 
\begin{equation}\label{go}
(\Ds X)^\sigma= \Ds (X^\sigma)\quad \mbox{for } s\in \R \mbox{ and }
\sigma\mbox{ a stopping time}.  
\end{equation}
Moreover, for $s\leq t\leq u$ we have
\begin{equation}\label{kl}
\Dt(\Ds X)_u = \Dt  X_u.
\end{equation}

\begin{definition} 
 Let $\mathcal{A}(\mcalF_{\bdot})$ denote the class of increasing
  adapted \ca\  processes.

 Let $\mathcal{A}^1(\mcalF_{\bdot})$ denote the subclass of
  $\mathcal{A}(\mcalF_{\bdot})$ consisting of integrable increasing
  \ca\  adapted
  processes;
  $\mathcal{LA}^1(\mathcal{F}_{\bdot})$ denotes  the subclass of
  $\mathcal{A}(\mcalF_{\bdot})$ consisting of \ca\ increasing  adapted
  processes $V=(V_t)_{t\in \R}$ for which there
  exists a localizing sequence $(\sigma_n)_{n\geq 1}$ such that
  $V^{\sigma_n}\in \mathcal{A}^1(\mcalF_{\bdot})$ for all~$n$. 

  Let $\mathcal{A}_0(\mcalF_{\bdot})$ denote the subclass of
  $\mathcal{A}(\mcalF_{\bdot})$ consisting of increasing \ca\ adapted
  processes $V=(V_t)_{t\in \R}$ for which $\lim_{t\to -\infty} V_t=0$
  $P$-a.s. Set $\mathcal{A}_0^1(\mcalF_{\bdot}) :=
  \mathcal{A}_0(\mcalF_{\bdot}) \cap\nobreak \mathcal{A}^1(\mcalF_{\bdot})$
  and $\mathcal{LA}_0^1(\mcalF_{\bdot}):=
  \mathcal{A}_0(\mcalF_{\bdot}) \cap \mathcal{LA}^1(\mcalF_{\bdot})$.

 Let $\mathcal{IA}(\mcalF_{\bdot})$ (resp.\
  $\mathcal{IA}^1(\mcalF_{\bdot})$,  $\mathcal{ILA}^1(\mcalF_{\bdot})$)  
  denote the class of \ca\ increasing processes
  $V$ for which $\Dsu V\in \mathcal{A}(\mcalF_{\bdot})$ 
 (resp.\ $\Dsu  V\in \mathcal{A}^1(\mcalF_{\bdot})$, $\Dsu  V\in
 \mathcal{LA}^1(\mcalF_{\bdot})$) for all $s\in \R$. We emphasize that  $V$
 is not assumed  adapted.  
\end{definition}

Motivated by our interest in increments  we say that 
two \ca\ processes 
$X=(X_t)_{t\in \R}$ and $Y=(Y_t)_{t\in \R}$    have \textit{identical increments},
and write $X\inl Y$, if $\Ds X  \Peq  \Dsu  Y$ 
for all $s\in \R$. In this case also  $X^\sigma\inl Y^\sigma$
whenever $\sigma$ is a stopping time. 

\begin{remark}\label{ff} Assume $X$ and $Y$ are \ca\ processes
  with $X\inl Y$. Then  by definition $X_t-X_s = Y_t- Y_s$ for all
  $s\leq t$ $P$-a.s.\ for all $t$  and so by the \ca\ property $X_t-X_s = Y_t- Y_s$ for all
  $s, t\in \R$ $P$-a.s. This shows that there exists a random variable
  $Z$ such that $X_t = Y_t +Z$ for all $t\in \R$ $P$-a.s., and thus
  $\Ds X_t = \Dsu  Y_t $ for all $s,t\in \R$ $P$-a.s.

\end{remark}

For any stochastic process  $X=(X_t)_{t\in \R}$  we have 
\begin{equation}\label{poi}
\Ds X_t + \Dt X_u = \Ds X_u \quad \mbox{for } s\leq
t\leq u. 
\end{equation}
This leads us to consider increment processes,  defined as
follows.  
Let $I=\{\Ds I\}_{s\in \R}$ with $\Ds I=(\Ds I_{t})_{t\in \R}$ be a family of
  stochastic processes. We say that $I$ is \textit{a
  consistent family
  of increment processes} if the following three conditions are
satisfied:  
\begin{enumerate}
\item[(1)] $\Ds I $ is an adapted process for all $s\in \R$, and $\Ds I_{t}=0$
  $P$-a.s.\ for all $t\leq s$. 
\item[(2)] For all $s\in \R$ and $\omega\in \Omega$  the mapping
  $t\mapsto \Ds I_{t}(\omega)$ is \ca.
\item[(3)]  For all $s\leq t \leq u$ we have $\Ds I_{t} + \Dt I_{u} =
  \Ds I_{u}$ 
  $P$-a.s. 
\end{enumerate}
 Whenever  $X$ is a \ca\ process such that  $\Ds X$ is
adapted for all $s\in \R$,   the family  
  $\{\Ds X\}_{s\in \R}$ of increment processes is then  consistent by equation
  \eqref{poi}. Conversely, 
let $I$ be a consistent family of increment processes.  A \ca\ 
process  $X=(X_t)_{t\in \R}$ is said to be \textit{associated with}
$I$  if $\Ds X\Peq  \Ds I$  for all $s\in \R$. It is easily seen that
there exists such a process; for example, let
\begin{equation*}
X_t = \begin{cases}
{{}^0\!}I_t & \mbox{for } t\geq 0 \\
 - \Dt I_0 &\mbox{for } t=-1, -2, \ldots, \\
X_{-n} +  {{}^{-n}\!}I_t  &\mbox{for } t\in (-n,-n+1) \mbox{ and } n=1,2,
\ldots
\end{cases} 
\end{equation*}
  Thus, consistent families of increment
processes correspond to increments in \ca\ processes with adapted
increments. If $X=(X_t)_{t\in \R}$ and $Y=(Y_t)_{t\in \R}$ are
\ca\ processes 
associated with $I$ then $X\inl Y$   and hence by Remark \ref{ff}
there is a random variable $Z$ such that $X_t = Y_t +Z$ for all $t$
$P$-a.s. 
\begin{remark}\label{nygo}  Let $I$ be a consistent family of increment
  processes, and assume     $X$ is a \ca\ process associated
  with $I$  such that 
$X_{-\infty}:=  \lim_{t\to -\infty}X_t$ exists in probability.  Then, 
$(X_t-X_{-\infty})_{t\in \R}$ is adapted  and associated with
  $I$. Indeed, $X_t - X_{-\infty}= \lim_{s \to -\infty} \Ds X_t$ in
  probability for $t\in \R$  and
  since $\Ds X_t= \Ds I_t$ ($P$-a.s.)   is $\mcalF_t$-measurable,
  it follows that
  $X_t- X_{-\infty}$ is $\mcalF_t$-measurable. In this case, 
  $(X_t-X_{-\infty})_{t\in \R}$  is the unique (up to
  $P$-indistinguishability) \ca\ process associated with $I$ which
  converges to $0$ in probability as time goes to $-\infty$. If, in 
  addition, $\Ds I$ is predictable for all $s\in \R$  then
  $(X_t-X_{-\infty})_{t\in \R}$  is also  predictable. To see this,  choose a
  $P$-null set $N$ and a sequence $(s_n)_{n\geq 1}$ decreasing to
  $-\infty$ such that $X_{s_n}(\omega)\to X_{-\infty}(\omega)$ as
  $n\to \infty$ for all $\omega\in N^c$. For $\omega\in N^c$
  and  $t\in \R$ we then have $X_t(\omega)- X_{-\infty}(\omega)=  \lim_{n
    \to \infty} {{}^{s_n}}X_{t}(\omega)$, implying the result due to
  inheritance of predictability  under pointwise limits.

\end{remark}

\section{Martingales and increment martingales} \label{bnm} 

Let us now introduce the classes   of (square integrable)  martingales and
the corresponding localized classes. 

\begin{definition} 
Let  $M=(M_t)_{t\in \R}$ denote a \ca\ adapted
process.  

We call   $M$
 an \textit{$\mcalF_{\bdot}$-martingale} if it  is
integrable and  
for all $s<t$,  $E[M_t|\mcalF_s]
=M_s$ $P$-a.s. If in addition $M_t$ is square integrable for all
$t\in \R$ then $M$ is called a \textit{ square integrable martingale}. 
Let $\mcalM(\mcalF_{\bdot})$ resp.\
$\mcalM^2(\mcalF_{\bdot})$ denote the class
of $\mcalF_{\bdot}$-martingales resp.\ square
integrable $\mcalF_{\bdot}$-martingales.  Note  that these
classes are both stable under stopping. 

We call  $M$  a \textit{local $\mcalF_{\bdot}$-martingale}  if there exists a
localizing sequence $(\sigma_n)_{n\geq 1}$ such that $M^{\sigma_n}\in
\mcalM(\mcalF_{\bdot})$ for all $n$. The definition of a
\textit{locally square integrable  martingale} is similar.  Let $\mathcal{L}\mcalM(\mcalF_{\bdot})$ resp.\
$\mathcal{L}\mcalM^2(\mcalF_{\bdot})$ denote the class of
{local martingales} resp.\ {locally square integrable
martingales}.  These classes  are  stable under
stopping. 

\end{definition}

\begin{remark}\label{-in}

(1) The  backward martingale convergence theorem shows that
if    $M\in \mcalM(\mcalF_{\bdot})$ then $M_t$
converges $P$-a.s.\ and in $L^1(P)$ to an $\mcalF_{-\infty}$-measurable
integrable random variable  $M_{-\infty}$   as $t\to -\infty$
(cf.\  Doob~(1990, Chapter~II, Theorem~2.3)).
In this case we may consider $(M_t)_{t\in [-\infty, \infty)}$ as a
martingale with respect to the filtration $(\mcalF_t)_{t\in
  [-\infty,\infty)}$.  If  $M\in \mcalM^2(\mcalF_{\bdot})$ then $M_t$
converges in $L^2(P)$ to $M_{-\infty}$.

(2) Let $M\in \mathcal{L}\mcalM(\mcalF_{\bdot})$ and choose a
localizing sequence $(\sigma_n)_{n\geq 1}$ such that $M^{\sigma_n}\in
\mcalM(\mcalF_{\bdot})$ for all $n$. From (1), it  follows that there
exists an $\mcalF_{-\infty}$-measurable integrable random variable
$M_{-\infty}$ (which does not depend on $n$) such that for all $n$ we
have $M_t^{\sigma_n} \to M_{-\infty} $ $P$-a.s.\ and in $L^1(P)$ as
$t\to -\infty$, and $M_t\to M_{-\infty}$ $P$-a.s. Thus, defining
$M^{\sigma_n}_{-\infty}:= M_{-\infty}$ it follows that for all $n$ the
process $\smash{(M_t)_{t\in [-\infty, \infty)}^{\sigma_n}}$ can be considered
a martingale with respect to $(\mcalF_t)_{t\in [-\infty,\infty)}$, and
consequently $\smash{(M_t)_{t\in [-\infty, \infty)}}$ is a local martingale.
(Note, though, that $\sigma_n$ is not allowed to take on the
value~$-\infty$.) In the case $M\in
\mathcal{L}\mcalM^2(\mcalF_{\bdot})$ assume $(\sigma_n)_{n\geq 1}$ is
chosen such that $ M^{\sigma_n}\in \mcalM^2(\mcalF_{\bdot})$ for all
$n$; then $M_t^{\sigma_n} \to M_{-\infty} $ in $L^2(P)$.

(3) The preceding shows that a local martingale indexed by $\R$ can
be extended to a  local martingale indexed by $[-\infty, \infty)$,
where localizing stopping times, however, are not allowed to take on
the value $-\infty$. Let us argue that the latter restriction is of
minor importance. Thus, call $\sigma: \Omega \to [-\infty, \infty]$ an
\textit{$\bar{\R}$-valued stopping time with respect to
  $\mathcal{F}_{\bdot}$} if $\{\sigma \leq t\} \in
\mcalF_t$ for all $t\in [-\infty, \infty)$, and call a sequence of
nondecreasing $\bar{\R}$-valued stopping times $\sigma_1\leq \sigma_2
\leq \cdots$ an \textit{$\bar{\R}$-valued localizing sequence} if
$\sigma_n \to \infty$ $P$-a.s.\ as $n\to \infty$. 

 Then we claim that a
\ca\ adapted process $M=(M_t)_{t\in \R}$ is a local martingale if and
only if $M_{-\infty} := \lim_{s\to -\infty} M_s$ exists $P$-a.s\ and
there is an $\bar{\R}$-valued localizing sequence $(\sigma_n)_{n\geq
  1}$ such that $(M_{t}^{\sigma_n})_{t\in [-\infty, \infty)}$ is a
martingale. We emphasize that the latter characterisation is the most
natural one when considering the index set $[-\infty, \infty)$, while
the former is better when considering $\R$. Note that the \textit{only
  if} part follows from (2). Conversely, assume $M_{-\infty} :=
\lim_{s\to -\infty} M_s$ exists $P$-a.s\ and let $(\sigma_n)_{n\geq
  1}$ be an $\bar{\R}$-valued localizing sequence such that
$(M_{t}^{\sigma_n})_{t\in [-\infty, \infty)}$ is a martingale, and let
us prove the existence of a localizing sequence $(\tau_n)_{n\geq 1}$
such that $M^{\tau_n}$ is a martingale for all $n$. Since
$M_{-\infty}$ is integrable it suffices to consider $M_t- M_{-\infty}$
instead of $M_t$; consequently we may and do assume $M_{-\infty}=0$.
In this case, $(\tau_n)_{n\geq 1} =(\tau \vee \sigma_n)_{n\geq 1}$
will do if $\tau$ is a stopping time such that $M^\tau$ is a
martingale. To construct this $\tau$ set $Z_t^n= E[|M_t^{\sigma_n} | |
\mcalF_{-\infty}]$ for $t\in [-\infty, \infty)$. Then $Z^n$ is
$\mcalF_{-\infty}$-measurable and can be chosen non-decreasing, \ca\
and $0$ at $-\infty$. Therefore
\begin{equation*}
\rho_n = \inf\{t\in \R: Z_{\frac{t}{2}}^n >1\} \wedge  0 
\end{equation*}
is real-valued, $\mcalF_{-\infty}$-measurable and $Z_{\rho_n}^n \leq
1$. Define
\begin{equation*}
\tau= \rho_n \wedge \sigma_n\  \mbox{on }  A_n = \{\sigma_1= \dots=
\sigma_{n-1} =- \infty  \mbox{ and } \sigma_n>-\infty \} 
\end{equation*}
and set  $\tau= 0$ on $(\cup_{n\geq 1} A_n)^c$. Then $\tau$ is a stopping
time since the $A_n$'s are disjoint and
$\mcalF_{-\infty}$-measurable. Furthermore,  $\cup_{n\geq 1} A_n =
\Omega$ $P$-a.s. Thus, for all $t> -\infty$,   
\begin{equation*}
E[|M_{t\wedge \tau}|] = \sum_{n=1}^\infty E[ |M_{\sigma_n \wedge \rho_n
  \wedge t} | 1_{A_n} ] = \sum_{n=1}^\infty E[ |Z^n_{\rho_n \wedge
  \sigma_n \wedge t}  | 1_{A_n} ] \leq 1,  
\end{equation*}
implying 
\begin{equation*}
E[ M_{\tau\wedge t} | \mathcal{F}_s] = \sum_{n=1}^\infty E[M_{\sigma_n
  \wedge \rho_n \wedge t} | \mathcal{F}_s]1_{A_n} = \sum_{n=1}^\infty
M_{\sigma_n \wedge \tau_n \wedge s} 1_{A_n} = M_{\tau\wedge s}
\end{equation*}
for all $-\infty < s<t$; thus,  $M^\tau$ is a martingale.

\end{remark}

\begin{example}\label{levdef} A \ca\ process $X=(X_t)_{t\in \R}$ is
  called a
  \textit{L\'evy process indexed by  $\R$} if it has stationary
  independent increments; that is,  whenever $n\geq 1$ and $t_0
  < t_1 <\dots <t_n$,  the increments
  ${{}^{t_0}\!}X_{t_1},{{}^{t_1}\!}X_{t_2},   \ldots,
  {{}^{t_{n-1}}\!}X_{t_n}$ are independent and  $\Ds X_t  \ilaw
  {{}^u\!}X_v$        whenever $s<t$ and  $u <v$
  satisfy $t-s = v-u$. In this case
  $(\Ds X_{s+t})_{t\geq 0}$ is an ordinary  L\'evy process indexed by 
  $\R_+$ for all $s\in \R$. 

Let  $X$ be  a L\'evy process indexed by  $\R$. There  is a unique 
infinitely divisible distribution $\mu$ on $\R$ associated with $X$
in the sense  that  for all
$s<t$, $\Ds X_t \ilaw \mu^{t-s}$, where,  for $u\geq 0$,  $\mu^u$ is the
probability measure with characteristic function $z\mapsto
\widehat\mu(z)^u$. (As always, $\widehat \mu$ denotes the
characteristic function of $\mu$).   When $\mu=N(0,1)$, the standard
normal  distribution, $X$ is called a (standard) Brownian motion
indexed by  $\R$. 
If $Y$ is a \ca\ process with $X\inl Y$, it  is a
L\'evy process as well and $\mu$ is also associated with $Y$; that is,    
L\'evy processes indexed by  $\R$ are  determined by the infinitely
divisible $\mu$  only up to  addition of a random variable.

Note that $(X_{(-t)-})_{t\in \R}$ (where, for $s\in \R$, $X_{s-}$
denotes the left limit at $s$) is again a L\'evy process
indexed by  $\R$  and
the distribution associated with it is $\mu^-$ given  by $\mu^-(B):=
\mu(-B)$ for $B\in \mathcal{B}(\R)$.
Since this process  appears by time reversion of $X$,  the behaviour of $X$ at
$-\infty$  corresponds to the behaviour of   $(X_{(-t)-})_{t\in \R}$ 
 at $\infty$, which is   well understood,  cf.\ e.g.\ Sato (1999,
 Proposition 37.10); in
 particular,   $\lim_{s\to -\infty}X_s$
 does not exist in $\R$ (in any reasonable sense) except when $X$ is
 constant.  Thus, except in  nontrivial cases  $X$ is not a local martingale
 in any filtration.

\end{example}    
This example clearly indicates that we need to generalise the concept
of a martingale.

\begin{definition}Let $M=(M_t)_{t\in \R}$ denote a \ca\ process,   in
  general  not assumed    adapted.

 We say that $M$ is an 
  \textit{increment martingale with respect to $\mathcal{F}_{\bdot}$} if for all $s\in \R$,  $\Ds M\in \mcalM(\mcalF_{\bdot})$. This is equivalent to
 saying that for all $s<t$, $\Ds M_t$ is $\mcalF_t$-measurable,
 integrable and satisfies $E[\Ds M_t  |\mcalF_s]=0$ $P$-a.s. If in
 addition all increments are square integrable, then $M$ is called  a
 \textit{increment square integrable martingale}. Let
 $\mathcal{IM}(\mcalF_{\bdot} )$ and
 $\mathcal{IM}^2(\mcalF_{\bdot})$ denote the corresponding
 classes. 

  $M$
is called an \textit{increment local martingale
    } if for all $s$,  $\Ds M $ is an adapted
  process and  there exists a localizing sequence $(\sigma_n)_{n\geq
    1}$ (which may depend on $s$) such that  $(\Ds M)^{\sigma_n}\in
  \mcalM(\mcalF_{\bdot})$ for all $n$. Define an 
  \textit{increment locally square integrable  martingale
   } in the obvious way. Denote the corresponding
  classes by $\mathcal{ILM}(\mcalF_{\bdot})$ and
 $\mathcal{ILM}^2(\mcalF_{\bdot})$.

\end{definition}
Obviously the four classes of increment processes are $\inl\,$-stable
and by  \eqref{go} stable under stopping. Moreover, 
$\mathcal{M}(\mcalF_{\bdot}) \subseteq \mathcal{IM}(\mcalF_{\bdot})$ and
$\mathcal{M}^2(\mcalF_{\bdot}) \subseteq \mathcal{IM}^2(\mcalF_{\bdot})$
with the following characterizations 
\begin{align} \label{k1}
\mathcal{M}(\mcalF_{\bdot})&=\{M=(M_t)_{t\in \R}\in \mathcal{IM}(\mcalF_{\bdot})
: M\ \text{is adapted and integrable}\}\\
\mathcal{M}^2(\mcalF_{\bdot})&=\{M\in \mathcal{IM}^2(\mcalF_{\bdot})
: M\ \text{is adapted and square integrable}\}. \label{k2}
\end{align}
Likewise,  
$\mathcal{LM}(\mcalF_{\bdot}) \subseteq \mathcal{ILM}(\mcalF_{\bdot})$
and  $\mathcal{LM}^2(\mcalF_{\bdot}) \subseteq \mathcal{ILM}^2(\mcalF_{\bdot})$. 
But no similar simple characterizations  as in \eqref{k1}--\eqref{k2}
of the localized classes   seem to be valid. Note 
that $\mathcal{LIM}(\mcalF_{\bdot})\subseteq \mathcal{ILM}(\mcalF_{\bdot})$, where
the former is the set of \textit{local increment martingales}, i.e.\   the
localizing  sequence can be chosen
independent of $s$. A similar statement holds for 
$\mathcal{ILM}^2(\mcalF_{\bdot})$.

When $\tau$ is a stopping time, we define $ {}^\tau\! M$ in the
obvious way as $ {}^\tau\! M_t= M_t -  M_{t\wedge \tau}$ for $t\in \R$. 
\begin{proposition}
  Let $M=(M_t)_{t\in \R}\in \mathcal{IM(F_{\bdot})}$ and $\tau$ be a stopping
  time with respect to $\mathcal{F}_{\bdot}$.  Then   ${}^\tau\! M\in \mathcal{M(F_{\bdot})}$ if
  $  \left\{M_0-M_{\tau\vee (-n)\wedge 0}:n\geq 1\right\}$ is
  uniformly integrable.
\end{proposition}
If  $\tau$ is bounded from below then the above set is always uniformly
integrable.
\begin{proof}
 Assume first that $\tau$ is bounded from below, that is,  there
 exists  an 
 $s_0\in (-\infty,0)$ such that $\tau\geq s_0$. Then, since  $ ({}^\tau\!
M_t)_{t\in
   \R}=({}^{s_0}\!M_t-{}^{s_0}\! M_{\tau\wedge t})_{t\in\R}$,  ${}^\tau\!
M$ is a sum of two martingales and hence a martingale. Assume now that
$  \left\{M_0-M_{\tau\vee (-n)\wedge 0}:n\geq 1\right\}$ is
  uniformly integrable. Then, with  $\tau_n=\tau\vee (-n)$ we have for
  $t\in \R$
$ {}^{\tau_n}\! M_t = (M_t- M_0) + (M_{\tau_n \wedge 0} - M_{\tau_n
  \wedge t}) + {}^{\tau_n}\! M_0$. The first term on the right-hand
side is integrable since  $M\in
\mathcal{IM(F_{\bdot})}$. Moreover, $ \{M_{\tau_n \wedge 0} - M_{\tau_n
  \wedge t} : n\geq 1\}$ is uniformly integrable  since
these random variables appear by stopping a martingale with bounded
stopping times. Thus, 
\begin{equation}
    \label{eq:4}
      \left\{{}^{\tau_n}\! M_t:n\geq 1\right\}\text{ is uniformly
        integrable for all }t\in \R.
  \end{equation} 
Since $\tau_n \uparrow \tau$ a.s.,  we have      ${}^{\tau_n}\! M_t\rightarrow
  {}^{\tau}\! M_t$ a.s.\  and    in $L^1(P)$ by \eqref{eq:4}. For all
  $n\geq 1$,  $\tau_n$ is  bounded from below
  and hence ${}^{\tau_n}\! M$ is a martingale, implying that
  ${}^{\tau}\! M$ is an $L^1(P)$-limit of martingales and hence a
  martingale.
\end{proof}
\begin{example}\label{qsx} Let $X=(X_t)_{t\in \R}$ denote a L\'evy
  process indexed by  $\R$. \textit{The filtration generated by the increments of
    $X$}  is   $\mcalF_{\bdot}^{\mathcal{I}X}=
  (\mcalF_t^{\mathcal{I}X})_{t\in \R}$, where 
\begin{align*}
\mcalF_t^{\mathcal{I}X} &= \sigma(\Ds X_t : s\leq t)\vee \mathcal{N} 
         =\sigma(\Ds X_u : s\leq u \leq t) \vee \mathcal{N}, \quad
         \mbox{for } t\in \R, 
\end{align*} 
and  we recall that $\mathcal{N}$ is the set of $P$-null sets. Using
a standard technique it can be verified that $\mcalF_{\bdot}^{\mathcal{I}X}$ is   a
filtration. Indeed, we only have to verify right-continuity of 
$\mcalF_{\bdot}^{\mathcal{I}X}$. 
 For this, fix $t\in \R$ and 
 consider random variables $Z_1$ and $Z_2$
 where $Z_1$ is bounded and
 $\mathcal{F}_t^{\mathcal{I}X}$-measurable, and  $Z_2$ is
 bounded and measurable with respect to $\sigma(\Ds X_u: t+\epsilon <
 s <u)$ for some $\epsilon >0$.   Then
 \begin{equation*}
 E[Z_1Z_2 | \mathcal{F}_{t+}^{\mathcal{I}X}] = Z_1 E[Z_2] = E[Z_1Z_2 |
 \mathcal{F}_{t}^{\mathcal{I}X}] \quad P\mbox{-a.s.}
 \end{equation*}
 by independence of  $Z_2$ and $\mathcal{F}_{t+}^{\mathcal{I}X}$. Applying the monotone class
 lemma it follows that whenever $Z$ is bounded and measurable with
 respect to  $\mcalF_\infty^{\mathcal{I}X}$  we have
 $E[Z|\mathcal{F}_{t+}^{\mathcal{I}X}] = E[Z|
 \mathcal{F}_t^{\mathcal{I}X} ]$ $P$-a.s., which in turn
 implies right-continuity of  $\mcalF_{\bdot}^{\mathcal{I}X}$.
It is readily seen that  $X\in
\mathcal{IM}(\mcalF_{\bdot}^{\mathcal{I}X})$ if $X$ has integrable centered increments. 
\end{example}

Increment martingales  are not necessarily integrable. But for
$M=(M_t)_{t\in \R} \in \mathcal{IM}(\mcalF_{\bdot})$,  $M_t\in L^1(P)$ for all
$t\in \R$ if and only if $M_t\in L^1(P)$ for some $t\in \R$.
Likewise $(M_s)_{s\leq t}$ is uniformly integrable for all  $t$ if and only if
$(M_s)_{s\leq t}$ is uniformly integrable for some $t$.
Similarly,  for $M\in \mathcal{IM}^2(\mcalF_{\bdot})$ we have  $M_t\in
L^2(P)$ for  all
$t\in \R$ if and only if $M_t\in L^2(P)$ for some $t\in \R$,  and
$(M_s)_{s\leq t}$ is $L^2(P)$-bounded for some $t$ if and only if
$(M_s)_{s\leq t}$ is $L^2(P)$-bounded  for some $t$.
For integrable
elements of $ \mathcal{IM}(\mcalF_{\bdot})$ we have the following
decomposition. 

\begin{proposition}\label{p3.7}  Let
$M=(M_t)_{t\in \R} \in \mathcal{IM}(\mcalF_{\bdot})$
  be integrable. Then $M$ can be decomposed uniquely up to
  $P$-indistinguishability  as $M=K +N$ where
  $K=(K_t)_{t\in \R} \in \mathcal{M}(\mcalF_{\bdot})$ and
  $N=(N_t)_{t\in \R} \in \mathcal{IM}(\mcalF_{\bdot})$ is an integrable
  process satisfying  
\begin{equation} \label{ml}
E[N_t| \mathcal{F}_t] =0 \mbox{ for all } t\in \R \quad \mbox{and}
\quad \lim_{t\to \infty} N_t=0 \ P\mbox{-a.s.\ and in } L^1(P). 
\end{equation}
If $M$ is square integrable then so are $K$ and $N$,  and $E[K_t
N_t]=0$ for all $t\in \R$. Thus $E[M_t^{2}]=E[K_t^2]+E[N_t^2]$
for all $t$ and moreover $t\mapsto E[N_t^2]$ is decreasing.
\end{proposition} 
  
\begin{proof}
The uniqueness is evident. To get the existence set 
$K_t = E[M_t| \mathcal{F}_t]$. Then $K$ is integrable and
adapted and for $s<t$ we have 
\begin{align*}
 E[K_t | \mathcal{F}_s] = E[M_t | \mathcal{F}_s]=E[M_s| 
\mathcal{F}_s] + E[\Ds M_t | \mathcal{F}_s]= K_s.
\end{align*}
Thus, $K \in \mathcal{M}(\mcalF_{\bdot})$ and therefore $N:= M-K \in
\mathcal{IM}(\mcalF_{\bdot})$. Clearly, $N$ is integrable and $E[N_t 
|  \mathcal{F}_t] =0$ for all $t\in \R$. Take $s\leq t$. Then
$\Ds N_t = E[\Ds N_t| \mathcal{F}_t]$, giving 
\begin{equation}\label{qwe}
\Ds N_t = E[N_t - N_s | \mathcal{F}_t] = -E[ N_s|
\mathcal{F}_t],
\end{equation}
that is $N_t = N_s - E[ N_s | \mathcal{F}_t]$,
proving that $\lim_{t\to \infty} N_t=0$ $P$-a.s.\ and in $ L^1(P)$. 
If $M$ is square integrable then so are  $K$ and $N$
and they are orthogonal. Furthermore for  $s\leq t$
\begin{align*}
  & E[N_s(N_t-N_s)]=E[(N_t-N_s)E[N_s|\mathcal{F}_t]]\\
& \qquad =E[(N_t-N_s)E[(N_s-N_t)|\mathcal{F}_t]]
=-E[(N_t-N_s)^2]
\end{align*}
implying
\begin{equation} \label{aze}
E[N_t^2]=E[N_s^2]-E[(N_t-N_s)^2].
\end{equation}
\end{proof}
As a corollary we may deduce the following convergence result for integrable
increment martingales.
\begin{corollary}\label{c3.8} Let
$M=(M_t)_{t\in \R} \in \mathcal{IM}(\mcalF_{\bdot})$
  be integrable. 
 \begin{itemize}
\item[(a)] If $(M_s)_{s\leq 0}$ is  uniformly integrable then
  $M_{-\infty}:=  \lim_{s\to -\infty} M_s$ exists $P$-a.s.\  and
in $L^1(P)$ and  $(M_t- M_{-\infty})_{t\in     \R}$  is in $\mcalM(\mcalF_{\bdot})$. 
\item[(b)] If $(M_s)_{s\leq 0}$ is bounded in $L^2(P)$ then  
$M_{-\infty}:= \lim_{s\to -\infty} M_s$ exists $P$-a.s.\  and
in $L^2(P)$ and  $(M_t- M_{-\infty})_{t\in     \R}$  is in $\mcalM^2(\mcalF_{\bdot})$.
\end{itemize}
 \end{corollary}
\begin{proof} 
Write $M=K+N$ as in Proposition \ref{p3.7}. As noticed  in Remark
\ref{-in}   
the conclusion holds for $K$. Furthermore
$(N_s)_{s\leq 0}$  is uniformly integrable when this 
is true for $M$ so we may and will assume $M=N$. That is, 
$M$ satisfies \eqref{ml}. By uniform integrability we can
find a sequence $s_n$ decreasing to $-\infty $ and an $\tilde{M}\in L^1(P)$
such that $M_{s_n}\to \tilde{M}$ in $\sigma (L^1,L^\infty)$. For all $t$
we have by  \eqref{qwe}
\begin{equation*}
M_t=M_{s_n}-E[M_{s_n}| \mathcal{F}_{t}]\quad  \text{for}\ s_n<t
\end{equation*}
 and thus
\begin{equation*}
M_t=\tilde{M}-E[\tilde{M}| \mathcal{F}_{t}]\quad  \text{for all}\ t,
\end{equation*}
proving part (a). In (b) the martingale part $K$ again has the right behaviour 
at $-\infty$. Likewise,  
$(N_s)_{s\leq 0}$ is bounded in $L^2(P)$ if this is true for $M$.
Thus we may assume that $M$ satisfies \eqref{ml}. The a.s.\ convergence is already
proved and the $L^2(P)$-convergence follows from \eqref{aze}  since
$t\mapsto E[M_t]$ is decreasing and $\sup _{s<0}E[M^2_s]<\infty $.
\end{proof}
Observe  that $(M_t- M_{t_0})_{t\in \R}$ is in $\mathcal{IM}(\mcalF_{\bdot})$ and
 is integrable
for every  $t_0\in \R$ and every $M\in \mathcal{IM}(\mcalF_{\bdot})$. Since a similar
result holds in the square integrable case,  
Corollary \ref{c3.8}  implies the following result relating convergence  of
an increment   martingale  to the 
martingale property.  

\begin{proposition}\label{zx} Let $M=(M_t)_{t\in \R}$
be a given \ca\ process.
The following are equivalent:   
\begin{itemize}
\item[(a)] $M_{-\infty}:= \lim_{s\to -\infty} M_s$ exists $P$-a.s.\  and
  $(M_t- M_{-\infty})_{t\in     \R}$  is in $\mcalM(\mcalF_{\bdot})$. 
\item[(b)]$M\in \mathcal{IM}(\mcalF_{\bdot})$ and  $(\Ds M_{0})_{s <0}$
  is uniformly integrable.  
\end{itemize}
Likewise, the following are equivalent:
\begin{itemize}
\item[(c)] 
$M_{-\infty} := \lim_{s\to -\infty} M_s$ exists $P$-a.s.\ 
  and  $(M_t- M_{-\infty})_{t\in     \R}$  is in $\mcalM^2(\mcalF_{\bdot})$
 \item[(d)]$M\in \mathcal{IM}^2(\mcalF_{\bdot})$ and $\sup_{s: s\leq 0} E[(\Ds M_0)^2] <\infty$.  
\end{itemize}
\end{proposition}

\begin{proof} Assuming $M\in
  \mathcal{IM}(\mcalF_{\bdot})/\mathcal{IM}^2(\mcalF_{\bdot})$, 
(b) $\Rightarrow$ (a) and (d) $\Rightarrow$ (c) follow  by
using Corollary 3.8 on $(M_t-M_0)_{t\in \R}$. The remaining two implications
follow  from standard martingale theory and the identity $\Ds
M_{0}=(M_0-M_{-\infty}) -(M_s-M_{-\infty})$.
\end{proof}

Let $M\in \mathcal{LM}(\mcalF_{\bdot})$ with
$M_{-\infty}=0$.  It is well-known that there 
exists a  unique (up to $P$-indistinguishability)  process  $[M]$
called \textit{the quadratic variation for $M$} 
satisfying   $[M]\in
\mathcal{A}_0(\mcalF_{\bdot})$, $(\Delta M)_t^2  = \Delta [M]_t$  for all
$t\in \R$ $P$-a.s.,    and $M^2- [M] \in
\mathcal{LM}(\mcalF_{\bdot})$. We have 
\begin{equation}\label{ee}
\Dsu  [M] \Peq  [\Ds M]\  \mbox{  for  }\  s\in \R \ \mbox{  and  }\  [M]^{\sigma}
\Peq [M^\sigma] \mbox{  when  } \sigma \mbox{ is a stopping time.} 
\end{equation}
If, in addition,  $M\in \mathcal{LM}^2(\mcalF_{\bdot})$,    there is a unique predictable process $\langle
M\rangle \in \mathcal{LA}^1_0(\mcalF_{\bdot})$ satisfying $M^2 - \langle
M\rangle \in \mathcal{LM}(\mcalF_{\bdot})$, and we shall call this process
the \textit{predictable quadratic variation  for  $M$}.  In this case,  
\begin{equation}\label{gg}
\Ds \langle M\rangle  \Peq  \langle \Ds M\rangle \  \mbox{  for  }\
s\in \R \ \mbox{  and  }\  \langle M\rangle^{\sigma}
\Peq \langle M^\sigma \rangle  \mbox{  when  } \sigma \mbox{ is a stopping time.} 
\end{equation}

\begin{definition}\label{godef} 

  Let $M\in \mathcal{ILM}(\mcalF_{\bdot})$.  We  say that  an increasing process
$V=(V_t)_{t\in \R}$ is a \textit{gene\-ralised quadratic   variation}
for $M$ if
\begin{gather}
\label{nr0}
V\in\mathcal{IA}(\mcalF_{\bdot})\\ \label{nr1}
(\Delta M)_t^2  = \Delta V_t  \quad   \mbox{for all }  t\in \R,\ 
P\mbox{-a.s.} \\
(\Ds M)^2 - \Dsu V \in \mathcal{LM}(\mcalF_{\bdot}) \quad  \mbox{for
  all }  s\in \R. \label{nr2} 
\end{gather}
We say that $V$ is \textit{quadratic variation for $M$} if, instead
of \eqref{nr0}, $V\in \mathcal{A}_0(\mcalF_{\bdot})$. 

 Let $M\in \mathcal{ILM}^2(\mcalF_{\bdot})$. We say that an
  increasing process $V=(V_t)_{t\in \R}$ is a \textit{generalised
    predictable quadratic variation }  for
  $M$ if  
\begin{gather}
V\in \mathcal{ILA}^1(\mcalF_{\bdot})  \label{uu1}\\
\Dsu V  \mbox{ is predictable for all } s\in \R\label{uu2}\\
(\Ds M)^2 - \Dsu  V \in \mathcal{LM}(\mcalF_{\bdot}) \quad  \mbox{for
  all }  s\in \R. \label{uu3} 
\end{gather}
We say that $V$ is a \textit{predictable quadratic variation}  for $V$ if, instead of
\eqref{uu1}, $V\in\nobreak \mathcal{LA}^1_0(\mcalF_{\bdot})$.
\end{definition}

\begin{remark} \label{lk}

(1) Let $M\in \mathcal{ILM}(\mcalF_{\bdot})$ and  $V$ denote a generalised
quadratic variation for $M$ such that $V_{-\infty}:= \lim_{s\to
  -\infty} V_s$ exists $P$-a.s. From Remark \ref{nygo} it  follows that
$(V_t -V_{-\infty})_{t\in \R}$ is a quadratic variation for $M$.

Similarly, let  $M\in \mathcal{ILM}^2(\mcalF_{\bdot})$ and  $V$ denote a
generalised predictable 
quadratic variation for $M$ such that $V_{-\infty}:= \lim_{s\to
  -\infty} V_s$ exists $P$-a.s. Then $(V_t -V_{-\infty})_{t\in \R}$ is
a  predictable quadratic variation for $M$. Indeed,  by   Jacod and
Shiryaev  (2003), Lemma I.3.10,
$( V_t - V_{-\infty})_{t\in \R}$  is a predictable process  in
$\mathcal{LA}_0^1(\mcalF_{\bdot})$. (Strictly speaking, this lemma 
  only ensures the existence of an $\bar{\R}$-value 
localizing sequence $(\sigma_n)_{n\geq 1}$ (cf.\ Remark \ref{-in}
(3))  such that $( V_t - V_{-\infty})^{\sigma_n}$ is in
$\mathcal{A}_0^1(\mcalF_{\bdot})$; this problem can, however, be dealt with
as described in   Remark \ref{-in}).

(2) If  $M\in \mathcal{LM}(\mcalF_{\bdot})$ with
$M_{-\infty}=0$  then the usual quadratic variation $[M]$ for $M$ is,
by \eqref{ee},  also a quadratic variation in the sense of Definition
\ref{godef}, and similarly, if  $M\in \mathcal{LM}^2(\mcalF_{\bdot})$ then  the
usual predictable quadratic variation $\langle M\rangle $ is a predictable quadratic variation also in the
sense defined above. 

(3) (Existence of generalised quadratic variation). Let $M\in \mathcal{ILM}(\mcalF_{\bdot})$. Then $V$ is a generalised
quadratic variation for $M$ if and only if we have \eqref{nr0}--\eqref{nr1} and $V$ is associated
with the family $\{[\Ds M]\}_{s\in \R}$. By Section \ref{sect2}, 
 existence and uniqueness (up to addition of random
variables) of the generalised quadratic
variation is thus ensured  once we have shown that the latter family is consistent. 
In other words, we must show for $s\leq t\leq u$ that 
$[\Ds M]_u = [\Ds M]_t + [\Dt M]_u$ 
$P$-a.s. Equivalently, $\Dt([\Ds M])_u = [\Dt M]_u$ $P$-a.s. This
follows,  however, from \eqref{ee} and \eqref{go}.

(4) (Existence of generalised predictable quadratic variation). Similarly,  let
$M\in \mathcal{ILM}^2(\mcalF_{\bdot})$. Then  $V$  is a
generalised   predictable quadratic variation for $M$ if and only if we have
\eqref{uu1}--\eqref{uu2} and $V$ is associated with $\{\langle \Ds M
\rangle\}_{s \in \R}$. Moreover, the latter family is consistent,
ensuring existence and uniqueness of the generalised predictable quadratic variation up to
addition of random variables.

(5) By Remark \ref{nygo}, the quadratic variation and the predictable quadratic variation are
unique up to $P$-indistinguishability when they exist.

(6) Generalised compensators and predictable compensators are
$\inl$-invariant,  i.e.\  if   for example $M, N\in \mathcal{IM}(\mcalF_{\bdot})$
with $M\inl N$ then $V$ is a generalised compensator for $M$ if and
only if it is a  generalised compensator for $N$.    

\end{remark}

When $M\in \mathcal{ILM}(\mcalF.)$ we use $[M]^\gen$ to denote a
generalised quadratic variation for $M$, and $[M]$ denotes the
quadratic variation when it exists. For $M\in
\mathcal{ILM}^2(\mcalF.)$, $\langle M\rangle ^\gen$ denotes a
generalised quadratic variation for $M$, and $\langle M\rangle $
denotes the predictable quadratic variation when it exists.
Generalising \eqref{ee}--\eqref{gg} we have the following.

\begin{lemma}
Let $\sigma$ denote a stopping time and $s\in \R$. If $M\in
\mathcal{ILM}(\mcalF_{\bdot})$ then 
\begin{equation}
([M]^\gen)^\sigma\inl [M^\sigma]^\gen\quad \mbox{and} \quad
\Ds([M]^\gen)\Peq 
[\Ds M]. \label{bx}
\end{equation}
If $M\in \mathcal{ILM}^2(\mcalF_{\bdot})$ then 
\begin{equation*}
(\langle M\rangle^\gen)^\sigma\inl \langle M^\sigma\rangle^\gen\quad
\mbox{and} \quad \Ds (\langle M\rangle^\gen)\Peq 
\langle \Ds M\rangle.
\end{equation*}
\end{lemma}

\begin{proof} We only prove the part concerning the quadratic
  variation. As seen above,  $[M]^\gen$ is associated with $\{[\Ds
  M]\}_{s\in \R}$,  which   implies  the second statement in \eqref{bx}. 

To prove the first statement in \eqref{bx} it suffices to show that
$([M]^\gen)^\sigma$ is associated with $\{[\Ds M^\sigma]\}_{s\in \R}$. 
Note that,  by 
\eqref{go} and \eqref{ee}, 
\begin{equation*}
\Ds(([M]^\gen)^\sigma) \Peq (\Dsu [M]^\gen)^\sigma   \Peq  [\Ds
M]^\sigma  \Peq 
[\Ds M^\sigma]. 
\end{equation*}  
\end{proof}

\begin{example} 

Let $\tau_1$ and $\tau_2$ denote independent  absolutely continuous random
variables with densities  $f_1$ and $f_2$  and distribution functions
$F_1$ and $F_2$ satisfying 
$F_i(t)<1$ for all $t$ and $i=1,2$.  Set 
\begin{equation*}
N_t^i= 1_{[\tau_i, \infty)}(t), \   A_t^i
=\int_{-\infty}^{t\wedge \tau_i} \tfrac{
  f_i(u)}{1-F_i(u)}\, \dd u, \  N_t=(N_t^1, N_t^1) \  \mbox{and}
\  \mcalF_t =\sigma(N_s: s\leq t)
\vee  \mathcal{N}
\end{equation*}
for $t\in \R$.   From   Br{\'e}maud  (1981), A2 T26, follows that
$(\mcalF_t)_{t\in \R}$ is right-continuous and hence a filtration in
the sense defined in the present paper. 
It is well-known  that $M^i$ defined by $M_t^i
= N_t^i - A_t^i$ is a square integrable martingale with $\langle
M^i\rangle_t = A_t^i$, and   
$M^1M^2$ is a martingale. 
Assume,  in addition, 
\begin{equation*}
\int_{-\infty}^t \tfrac{u f_i(u)}{1-F_i(u)} \, \dd u
=-\infty \quad \mbox{for all } t\in \R. 
\end{equation*}
(This is satisfied if, for example, $F_i(s)$
equals  a constant times $(1 +|s|\log(|s|))^{-1}$ when  $s$ is small.)
Let $B^i \in \mathcal{IA}^1(\mcalF_{\bdot})$ satisfy 
\begin{equation*}
\Ds B_t^i = \int_{(s,
  t]} u \, \dd A_u^i = \int_{s \wedge \tau_i}^{t\wedge \tau_i} \frac{u
  f_i(u)}{1-F_i(u)}\, \dd u
\end{equation*}
 for $s<t$  and set  $X_t^i = \tau_i
N_t^i- B_t^i$. Then   
\begin{equation*}\label{bh}
\lim_{s\to -\infty} X_s^i = -\lim_{s\to
  -\infty} B_s^i =\infty \quad \mbox{pointwise,} 
\end{equation*}
implying  that  $X^i$ is not a local
martingale. However, since for $s<t$,  
\begin{equation*}
\Ds X_t^i =\int_{(s,t]} u \, \dd M_u^i
\end{equation*}
it follows that   $\Ds X^i$ is a square
integrable martingale. That is, $X^i\in \mathcal{ILM}^2(\mcalF_{\bdot})$. 

The quadratic  variations,  $[X^i]$ resp.\ $[X^1 -X^2]$, of
$X^i$ resp.\ $X^1-X^2$ do exist and  are $[X^i]_t = (\tau_i)^2N_t^i$
resp.\  $[X^1 - X^2]_t = (\tau_1)^2N_t^1 + (\tau_2)^2N_t^2$. 
Moreover, up to addition of  random variables,  
\begin{alignat*}{2}
\liminf_{s\to   -\infty} (X_s^1 - X_s^2)&=\liminf_{s\to   -\infty}
(B_s^2 - B_s^1)&&= \liminf_{s\to -\infty} \int_s^0 u (\tfrac{f_2(u)}{1-
F_2(u)}- \tfrac{f_1(u)}{1-F_1(u)}) \, \dd u
\\
\limsup_{s\to   -\infty} (X_s^1 - X_s^2)&=\limsup_{s\to   -\infty}
(B_s^2 - B_s^1)&&= \limsup_{s\to -\infty} \int_s^0 u (\tfrac{f_2(u)}{1-
F_2(u)}- \tfrac{f_1(u)}{1-F_1(u)}) \, \dd u.
\end{alignat*}
 If $\tau_1$ and $\tau_2$ are identically distributed 
then $X_s^1- X_s^2$ converges pointwise.  In other cases we may have     $\limsup_{s\to -\infty}
(X_s^1- X_s^2) = -\liminf_{s\to -\infty} (X_s^1 - X_s^2) =\infty$
pointwise. 

To sum up,  we have seen  that even if  the quadratic variation 
  exists, the   process may or
  may not converge as time goes to $-\infty$. 

\end{example}

The next result shows in particular that for increment  local  
martingales with bounded jumps, a.s.\  convergence at   $-\infty$
is closely related to the local martingale property.

\begin{theorem}\label{vbb} Let $M\in \mathcal{ILM}^2(\mcalF_{\bdot})$. The
  following are equivalent.
\begin{itemize}
\item[(a)]  There is a   predictable quadratic variation  $\langle M \rangle  $ for $M$.
\item[(b)] $M_{-\infty}= \lim_{s\to -\infty} M_s$ exists $P$-a.s.\ and
  $(M_t - M_{-\infty})_{t\in \R} \in \mathcal{LM}^2(\mcalF_{\bdot})$. 
\end{itemize}
\end{theorem}

\begin{remark}   Let $M$ in
  $\mathcal{ILM}(\mcalF_{\bdot})$ have bounded jumps; then,    $M\in
  \mathcal{ILM}^2(\mcalF_{\bdot})$ as   well. In this case  (b) is satisfied
  if and only   if $M_{-\infty}:= \lim_{s\to -\infty}
  M_s$ exists $P$-a.s. Indeed,  if the limit exists 
   we define   
\begin{equation*}
\sigma_n =\inf\{t \in \R: |M_t-M_{-\infty}| >n\}. 
\end{equation*}
Then $(M_t^{\sigma_n}- M_{-\infty})_{t\in \R}$ is a bounded and  adapted
process  in
$\mathcal{ILM}(\mathcal{F}_{\bdot})$ and hence in
$\mathcal{IM}^2(\mathcal{F}_{\bdot})$. By Proposition \ref{zx},
$(M_t^{\sigma_n}- M_{-\infty})_{t\in \R}$ is in
$\mathcal{M}^2(\mcalF_{\bdot})$.  
\end{remark}

\begin{proof}  (a) implies (b):  Choose a localizing sequence $(\sigma_n)_{n\geq 1}$
such that  
\begin{equation*}
E[\langle M\rangle^{\sigma_n}_t ]< \infty, \quad \mbox{for all }
t\in \R \mbox{ and all } n\geq 1.
\end{equation*}
Since $\Ds \langle M\rangle^{\sigma_n} = \langle \Ds M^{\sigma_n}
\rangle$, it follows in particular that  
\begin{equation*}
E[\langle \Ds M^{\sigma_n}\rangle_t] \leq E[\langle
M\rangle^{\sigma_n}_t ]< \infty 
\end{equation*}
for all $s\leq t$ and $n$. Therefore, for all $s$ and $n$ we have  $\Ds
M^{\sigma_n} \in \mathcal{M}^2(\mcalF_{\bdot})$,  and
\begin{equation*}
 E[(\Ds M^{\sigma_n}_t)^2] \leq  E[\langle
M\rangle^{\sigma_n}_t] <\infty   
\end{equation*}
for all $s\leq t$.   Using Proposition \ref{zx}  on $M^{\sigma_n}$ it
follows that $M_{-\infty}:= \lim_{s\to -\infty} M_s^{\sigma_n}$ exists
$P$-a.s.\ (this limit does not depend on $n$) and $(M_t^{\sigma_n} -
M_{-\infty})_{t\in\R}$ is a square integrable martingale.  

(b) implies (a): Let $\langle M - M_{-\infty}\rangle $ denote the
predictable quadratic variation for  $(M_t-M_{-\infty})_{t\in \R}$ which
exists since this process is a locally square integrable
martingale. Since $M\inl (M_t-M_{-\infty})_{t\in \R}$,   $\langle M -
M_{-\infty}\rangle $ is a  predictable quadratic variation for $M$ as
well.

\end{proof}

We have seen that   a continuous increment local  martingale is a
local martingale if it converges
almost surely as time goes to $-\infty$.   A main  purpose of the  next examples  is to study the
behaviour at    $-\infty$ when this is not the case.

\begin{example} In (2) below we give an example of a continuous 
  increment local martingale which converges  to zero in probability as time
  goes to $-\infty$ without being  a local martingale. As a building
  block for this construction we first consider  a  simple
  example of a continuous local martingale which is nonzero only on a
  finite interval. 

(1)   Let
$B=(B_t)_{t\geq 0}$ denote a standard Brownian motion and  $\tau$
be  the first visit to zero after a visit to $k$, i.e.
\begin{equation}\label{kdef}
\tau= \inf\{t>0: B_t=0 \mbox{ and there is an } s<t \mbox{ such that }
  B_s>k\},
\end{equation}
where  $k>0$ is some fixed level. Then  $\tau$ is finite with
probability one, the stopped process $(B_{t\wedge \tau})_{t\geq 0}$ is
a square integrable martingale,  and $B_{t\wedge \tau}=0$ when $t\geq
\tau$. Let $a<b$ be real numbers and $\phi: [a,b) \to
 [0,\infty)$ be a surjective, continuous and strictly increasing
 mapping   and define $Y=(Y_t)_{t\in \R}$ as 
\begin{equation}\label{cv}
Y_t = \begin{cases}
       0 & \mbox{if } t<  a\\
       B_{\phi(t) \wedge \tau}  & \mbox{if } t\in [a, b)\\
       0  & \mbox{if } t\geq b. 
\end{cases}
\end{equation}
Note that $t\mapsto Y_t $ is continuous $P$-a.s.\ and that with
probability one $Y_t=0$ for $t\not\in [a,b]$. Define, with
$\mathcal{N}$ denoting the $P$-null sets,  
\begin{equation}\label{hj}
\mcalF_t=  \sigma(B_u: u\leq \phi(t) ) \vee \mathcal{N} \quad \mbox{
  for } t\in \R, 
\end{equation}
where we let $\phi(t)=0$ for $t\leq a$ and $\phi(t)= \infty$ for
$t\geq b$.  Interestingly,  $Y$ is  
a local martingale. To see this,  define the "canonical"\   localizing
sequence $(\sigma_n)_{n\geq 1}$ as $\sigma_n =\inf\{t\in \R: |Y_t|
>n\}$. 
Since $(Y^{\sigma_n}_t)_{t\in [a,b)}$ is  a deterministic time change of
$(B_{t\wedge \tau})_{t\geq 0}$ stopped at $\sigma_n$, it  is a
bounded, and hence uniformly integrable,  martingale. By continuity of
the paths and the property $Y^{\sigma_n}_t = Y^{\sigma_n}_b$ for
$t\geq b$  it thus follows that  $(Y^{\sigma_n}_t)_{t\in \R}$ is a
bounded martingale.

(2) For $n=1,2, \ldots$ let $B^n=(B^n_t)_{t\geq 0}$ denote independent
standard Brownian motions, and define $Y^n=(Y^n_t)_{t\in \R}$ as in
\eqref{cv} with $a=-n$ and $b=-n+1$, and $Y$ resp.\ $B$ replaced by
$Y^n$ resp.\ $B^n$.  Let $(\mathcal{F}_t^n)_{t\in \R}$ be the
corresponding filtration defined as in \eqref{hj}, and $(\theta_n)_{n
  \geq 1}$ denote a sequence of independent Bernoulli variables that
are independent of the Brownian motions as well and satisfy
$P(\theta_n=1)= 1- P(\theta_n=\nobreak 0)= \tfrac{1}{n}$ for all $n$. Let
$X_t^n=\theta_nY^n_t$ for $t\in \R$.

Define $X_t =\sum_{n=1}^\infty X_t^n$ for $t\in \R$, which is
well-defined since $X_t^n=0$ for $t\not\in [-n,-n+1]$, and set
$\mcalF_t = \vee_{n=1}^\infty (\mathcal{F}_t^n \vee \sigma(\theta_n))$
for $t\in \R$. For $s\in [-n,-n+1]$ and $n=1,2, \ldots$, $\Ds X_t =
\sum_{m=1}^n \Ds X_t^m$, and since it is easily seen that each
$(X_t^m)_{t\in \R}$ is a local martingale with respect to
$(\mcalF_t)_{t\in \R}$, it follows that $\Ds X$ is a local martingale
as well; that is, $X$ is an increment local martingale.  By
Borel-Cantelli, infinitely many of the $\theta_n$'s are $1$ $P$-a.s.,
implying that $X_s$ does not converge $P$-a.s.\ as $s\to -\infty$. On
the other hand, $P(X_t=0)\geq \tfrac{n-1}{n}$ for $t\in [-n, -n+1]$,
which means that $X_s\to 0$ in probability as $s\to -\infty$.

From  \eqref{k1} it follows that  if a process in
$\mathcal{IM}(\mcalF_{\bdot})$ is adapted and integrable then it is in
$\mathcal{M}(\mcalF_{\bdot})$. By the above  there is no such result
for $\mathcal{ILM}(\mcalF_{\bdot})$; indeed, $X$ is both adapted and
$p$-integrable for all $p>0$ but it is not in
$\mathcal{LM}(\mcalF_{\bdot})$.    

\end{example}

\begin{example}\label{ivba}

Let $X=(X_t)_{t\geq 0}$ denote the inverse of  $\textrm{BES}(3)$,
the three-dimensional Bessel process. It is well-known (see e.g.\  Rogers and
Williams (2000))
that $X$ is a diffusion on natural scale  and hence for  all $s>0$ 
the increment process 
$(\Ds X_t)_{t\geq 0}$ is a
local martingale. That is,  we may consider $X$ as an increment  
  martingale indexed  by 
$[0,\infty)$.   By Rogers and Williams (2000),
$\infty$ is   an  entrance boundary, which  means    that  if the
process is started in   $\infty$, it
immediately leaves this state and  never returns. 
Since we can obviously stretch $(0, \infty)$ into $\R$,  this shows that
there are  interesting examples of continuous  increment local martingales
$(X_t)_{t\in \R}$ for which $\lim_{t\to -\infty} X_t = \pm \infty$
almost surely.

\end{example}

Using the Dambis-Dubins-Schwartz theorem it follows easily that any
continuous local martingale indexed by~$\R$ is a time change of a
Brownian motion indexed by~$\R_+$. It is not clear to us whether there
is some analogue of this result for continuous increment local
martingales but there are indications that this it not the case;
indeed, above we saw that a continuous increment local martingale may
converge to $\infty$ as time goes to $-\infty$; in particular this
limiting behaviour does not resemble that of a Brownian motion indexed
by $\R_+$ as time goes to $0$ or of a Brownian motion indexed by $\R$
as time goes to $-\infty$.

Let $M\in \mathcal{LM}(\mcalF_{\bdot})$. It is well-known that  $M$ can be decomposed
uniquely up to $P$-indistinguishability as $M_t = M_{-\infty} +M_t^c + 
M_t^d$ where $M^c=(M_t^c)_{t\in \R}$,  the \textit{continuous
  part} of $M$,  is a continuous local martingale with
$M_{-\infty}=0$, and $M^d$, the \textit{purely discontinuous part} of
$M$,  is a purely discontinuous local martingale with $M^d_{-\infty}=0$,
which means that $M^dN$ is a local martingale for all continuous local
martingales $N$. Note that for $s\in \R$, 
\begin{equation} \label{rrt}
(\Ds M)^c = \Ds (M^c) \quad \mbox{ and } \quad  (\Ds M)^d = \Ds (M^d). 
\end{equation}
We need a further decomposition of $M^d$ so   let $\mu^M=\{\mu^M(\omega;
\dd t, \dd x):\omega \in \Omega\}$ denote the random  measure on
$\R\times (\R\setminus\{0\})$ induced by the  jumps of $M$; that is, 
\begin{equation*}
\mu^M(\omega; \dd t , \dd x)= \sum_{s\in \R} \delta_{(s, \Delta
  M_s(\omega))}(\dd t, \dd x), 
\end{equation*}  
and let  $\nu^M=\{\nu^M(\omega;
\dd t, \dd x):\omega \in \Omega\}$ denote the compensator  of $\mu^M$
in the sense of Jacod and Shiryaev (2003), II.1.8.  From Proposition
II.2.29 and  Corollary II.2.38 
in Jacod and Shiryaev (2003) it follows that   $(|x| \wedge |x|^2)
*\nu^M \in \mathcal{LA}_0^1(\mcalF_{\bdot})$  
and $M^d \Peq  x*(\mu^M-\nu^M)$, implying that   
  for arbitrary $\epsilon >0$, $M$ can be decomposed as  
\begin{align*}
M_t &= M_{-\infty} + M_t^c +M_t^d = M_{-\infty} + M_t^c + x* (\mu^M -
\nu^M)_t\\
    &=  M_{-\infty}+  M_t^c +(x1_{\{|x|\leq \epsilon\}}) *(\mu^M -
    \nu^M)_t +(x1_{\{|x|>  \epsilon\}}) *\mu^M_t -(x1_{\{|x|>
      \epsilon\}}) *\nu^M_t   .
\end{align*}
Recall that when $M$ is quasi-left continuous we have 
\begin{equation}\label{rty}
\nu^M(\cdot; \{t\} \times (\R\setminus \{0\})) = 0 \quad \mbox{for all
} t\in \R\ P\mbox{-a.s.} 
\end{equation}
Finally, for $s\in \R$, $\mu^{\Ds M}(\cdot; \dd t, \dd x) =1_{(s,
  \infty)}(\dd t) \mu^M(\cdot; \dd t , \dd x)$ and thus 
\begin{equation}\label{ion}
\nu^{\Ds M}(\cdot; \dd t, \dd x) =1_{(s,
  \infty)}(\dd t) \nu^M(\cdot; \dd t , \dd x).
\end{equation}

Now consider the case $M\in \mathcal{ILM}(\mcalF_{\bdot})$. Denote the
continuous resp.\ purely discontinuous part of $\Ds M$ by $\Ds M^c$
resp.\ $\Ds M^d$. By \eqref{rrt}, $\{\Ds M^c\}_{s\in \R}$ and $\{\Ds
M^d\}_{s\in \R}$ are consistent families of increment processes, and
$M$ is associated with $\{\Ds M^c +\nobreak \Ds M^d\}_{s\in \R}$. Thus, there
exist two processes, which we call the \textit{continuous} resp.\
\textit{purely discontinuous} part of $M$, and denote $M^{c\gen}$ and
$M^{d\gen}$, such that $M^{c\gen}$ is associated with $\{\Ds
M^c\}_{s\in \R}$ and $M^{d\gen} $ is associated with $\{\Ds
M^d\}_{s\in \R}$, and
\begin{equation} \label{psx}
M_t = M_t^{c\gen} + M_t^{d\gen}\quad \mbox{for all } t\in\R,\ 
P\mbox{-a.s.}
\end{equation}
Once again these processes are unique only up to addition of random
variables.  In view of \eqref{ion} we define the compensator of
$\mu^M$, to be denoted $\{\nu^M(\omega;\dd t , \dd x):\nobreak \omega
\in\nobreak \Omega\}$, as the random measure on $\R\times
(\R\setminus\{0\})$ satisfying that for all $s\in \R$, 
\[
1_{(s, \infty)
}(\dd t) \nu(\omega; \dd t, \dd x)=\nu^{\Ds M} (\omega; \dd t, \dd
x),
\]
where, noticing that $\Ds M$ is a local martingale, the right-hand
side is the compensator of $\mu^{\Ds M}$ in the sense of Jacod and
Shiryaev (2003), II.1.8.

\begin{theorem} Let $M\in \mathcal{ILM}(\mcalF_{\bdot})$. \label{dbb}
\begin{itemize}
\item[(1)] The quadratic variation $[M]$   for $M$ exists  if and only
  if there
  is a continuous martingale component $M^{c\gen}$ with $M^{c\gen}  \in
  \mathcal{LM}(\mcalF_{\bdot})$ and $M_{-\infty}^{c\gen} =0$, and for all $t\in
  \R$,  $\sum_{s\leq t} (\Delta M_s)^2 <\infty$ $P$-a.s. In this case 
\begin{equation*}
[M]_t = \langle M^{c\gen}\rangle_t + \sum_{s\leq t} (\Delta M_s)^2 .
\end{equation*}
\item[(2)] We have that $M_{-\infty}:= \lim_{s\to -\infty} M_s$ exists
  $P$-a.s.\ and   $(M_t - M_{-\infty})_{t\in
    \R}\in   \mathcal{LM}(\mcalF_{\bdot})$ if
  and only if the quadratic variation $[M]$   for $M$ exists and
  $[M]^{\frac{1}{2}} \in \mathcal{LA}_0^1(\mcalF_{\bdot})$. 
\item[(3)] Assume \eqref{rty} is satisfied and there is an $\epsilon>0$
 such that 
\begin{equation} \label{iii}
\lim_{s\to -\infty}\int_{(s, 0]}\int_{|x|>\epsilon} x \nu^M(\cdot;
\dd u, \dd x)
\end{equation}
exists  $P$-a.s.  Then, 
$\lim_{s\to -\infty} M_s$ exists $P$-a.s.\ if     and only if $[M]$ exists.  
\end{itemize}
\end{theorem}
Note that the conditions in (3) are satisfied if   $\nu^M$ can be decomposed as
$\nu^M(\cdot; \dd t \times \dd x)=
F(\cdot;t,  \dd x)\, \mu(\dd t)  $ where $F(\cdot;t,  \dd x)$ is a symmetric
measure for all   $t\in \R$ and $\mu$ does not have positive point masses. 

\begin{proof} (1)  For  $s\leq t$ we have 
\begin{align} \nonumber 
\Dsu  [M]^\gen_t = [\Ds M]_t &= \sum_{u: s< u\leq t} (\Delta M_u)^2 +
\langle \Ds M^c \rangle_t  \nonumber         \\ 
&= \sum_{u: s< u\leq t} (\Delta M_u)^2 + \langle\Ds (  M^{c\gen})
 \rangle_t  \nonumber  \\
 &= \sum_{u: s< u\leq t} (\Delta M_u)^2 +\Ds \langle  M^{c\gen}
 \rangle_t^{\gen}   \nonumber \\
 &=  \sum_{u: s< u\leq t} (\Delta M_u)^2 
 + \langle  M^{c\gen}
 \rangle_t^{\gen} - \langle  M^{c\gen}
 \rangle_s^{\gen}, \label{vvvv}
\end{align}
where the first equality is due to the fact that $[M]^\gen$ is
associated with $\{[\Ds M]\}_{s\in \R}$, the second is a well-known
decomposition of the quadratic variation of a local martingale,  the third
equality is due to $M^{c\gen}$ being associated with $\{\Ds
M^c\}_{s\in \R}$ and the fourth is due to $\langle
M^{c\gen}\rangle^\gen $
being  associated with
$\{\langle \Ds M^{c\gen}\rangle\}_{s\in \R}$. 
By Remark \ref{lk} (1), the quadratic variation $[M]$ exists if and only
if $[M]_s^\gen$ converges $P$-a.s.\ as $s\to -\infty$, which, by the
above, is equivalent to convergence almost surely of both terms in
\eqref{vvvv}.  By Theorem \ref{vbb}, $\langle M^{c\gen}
\rangle_s^\gen$ converges $P$-a.s.\ as $s\to -\infty$ if and only if $M_{-\infty}^{c\gen}$
exists $P$-a.s.\ and $(M^{c\gen}_t - M_{-\infty}^{c\gen})_{t\in \R}$
is a continuous local martingale. If the quadratic variation exists,
we may replace  $M^{c\gen}$ by
$(M^{c\gen}_t - M_{-\infty}^{c\gen})_{t\in \R}$ and  $M^{d\gen}$ by 
$(M^{d\gen}_t + M_{-\infty}^{c\gen})_{t\in \R}$, thus obtaining a
continuous part of $M$ which starts at $0$.

(2) First assume that $M_{-\infty}$ exists  and
$(M_t-M_{-\infty})_{t\in
  \R}\in \mathcal{LM}(\mcalF_{\bdot})$. Since $M\inl (M_t -
M_{-\infty})_{t\in \R}$,  the quadratic variation for  $M$ exists and equals
the quadratic variation for  $(M_t-M_{-\infty})_{t\in
  \R}$. It is well-known that since the latter is a local martingale,
$[M]^{\frac{1}{2}} \in \mathcal{LA}_0^1(\mcalF_{\bdot})$. 

Conversely assume that  $[M]$  exists and  $[M]^{\frac{1}{2}} \in
\mathcal{LA}_0^1(\mcalF_{\bdot})$. Choose a localizing sequence
$(\sigma_n)_{n\geq 1}$ such that $[M^{\sigma_n}]^{\frac{1}{2}}  \in
\mathcal{A}_0^1(\mcalF_{\bdot})$. Since   $\Dsu  [M^{\sigma_n}]_0 \leq
[M^{\sigma_n}]_0 $ if follows from Davis' inequality that for some
constant $c>0$,  
\begin{equation*}
E[\sup_{u: s\leq u\leq 0} |\Ds M_0^{\sigma_n}| ] \leq c E[
[M^{\sigma_n}]_0^{\frac{1}{2}}] <\infty 
\end{equation*}
for all $s\leq 0$,  implying that $(\Ds M_0^{\sigma_n} )_{s<0}$ is
uniformly integrable.  The result now follows from Proposition~\ref{zx}.

(3) By \eqref{ion}, the three families of increment processes $\{
(x1_{\{|x| \leq \epsilon\}}) *(\mu^{\Ds M} -\nobreak \nu^{\Ds
  M}\}_{s\in \R}$, $\{ (x1_{\{|x|> \epsilon\}}) * \mu^{\Ds M} \}_{s\in
  \R}$ and $\{ (x1_{\{|x|> \epsilon\}}) * \nu^{\Ds M} \}_{s\in \R}$
are all consistent.  Choose $X=(X_t)_{t\in \R}$, $Y=(Y_t)_{t\in \R}$
and $Z=(Z_t)_{t\in \R}$ associated with these families such that $X_t
+Y_t -Z_t = M_t^{d\gen}$; in particular we then have $ M \Peq
M^{c\gen} + X +Y -Z $.  Since $Z$ is associated with $\{ (x1_{\{|x|>
  \epsilon\}}) * \nu^{\Ds M} \}_{s\in \R}$ we have
\begin{equation*}
Z_{0} -
Z_{s}   =\int_{s}^0\int_{|x|>\epsilon}   x\,      \nu^{M} (\cdot;
\dd u,      \dd x) \quad \mbox{for all } s\in \R \mbox{ with
  probability one,} 
\end{equation*}
 implying that $s\mapsto Z_s$ is continuous by \eqref{rty} and
$\lim_{s\to -\infty} Z_s$ exists $P$-a.s.\ by \eqref{iii}.  By
\eqref{rty} it also follows that 
$
(\Delta X_s)_{s\in \R} \Peq
(\Delta M_s 1_{ \{|\Delta M_s|\leq \epsilon\}})_{s\in \R} 
$,
implying that $X$ is an  increment local  martingale with jumps bounded by
$\epsilon$ in absolute value and 
\begin{equation} \label{tyv}
\sum_{s:s\leq t } (\Delta M_s)^2 = \sum_{s:s\leq t} ( \Delta  X_s)^2 +
\sum_{s:s\leq t}(\Delta Y_s)^2\quad \mbox{for all } t\in \R \mbox{
  with probability one.}
\end{equation}

If $[M]$ exists then by (1) $M^{c\gen}_{-\infty}$ exists $P$-a.s.\ and
\eqref{tyv} is finite for all $t$ with probability one. Since $Y$ is
piecewise constant with jumps of magnitude at least~$\epsilon$, it
follows that $Y_s$ is constant when $s$ is small enough almost surely.
In addition, since the quadratic variation of the increment local
martingale $X$ exists and $X$ has bounded jumps it follows from (2)
that, up to addition of a random variable, $X$ is a local martingale
and thus $\lim_{s\to -\infty} X_s$ exists as well; that is,
$\lim_{s\to -\infty} M_s$ exists $P$-a.s.

If,  conversely, $\lim_{s\to -\infty} M_s$ exists $P$-a.s.,   there
are no jumps of magnitude at least $\epsilon$ in $M$ when $s$ is small
enough; thus there are no jumps in $Y_s$ when $s$ is sufficiently
small  $P$-a.s., implying that
$\lim_{s\to -\infty} (M_s^{c\gen} +X_s)$ exists $P$-a.s. Combining
Theorem~\ref{vbb}, \eqref{tyv}  and (1) it follows  that $[M]$ exists.  
\end{proof}

\section{Stochastic integration} 

In the following we define a stochastic integral with respect to an
increment local martingale. 
Let  $M\in \mathcal{LM}(\mcalF_{\bdot})$ and  set 
\begin{align*}
&\mathcal{L}L^1(M)\\ \quad 
&:= \{\phi= (\phi_t)_{t\in \R}: \phi \mbox{ is predictable and } 
\Big(\big(\int_{(-\infty, t]} \phi^2_s \, \dd[M]_s\big)^{\tfrac{1}{2}}\Big)_{t\in \R}  \in \mathcal{LA}_0^1(\mcalF_{\bdot}) \}.
\end{align*}
Since in this case   the index set  set can be taken to be $[-\infty, \infty)$, it is
well-known, e.g.\ from Jacod (1979), 
that the stochastic integral of $\phi\in \mathcal{L}L^1(M)$ with
respect to $M$,  which we denote $
(\int_{(-\infty, t]} \phi_s    \, \dd M_s)_{t\in \R}$ or
$\phi\bullet M= (\phi\bullet M_t)_{t\in \R}$,  does exist. All
fundamental properties of the integral   are well-known so let us just
explicitly mention the following two results that we are going to use
in the following:   
For $\sigma$ a stopping time,  $s\in \R$ and $\phi \in
    \mathcal{L}L^1(M)$   we have 
\begin{equation}\label{nimm1}
(\phi \bullet M)^\sigma \Peq  (\phi 1_{(-\infty, \sigma]})\bullet
M \Peq 
\phi \bullet (M^\sigma) 
\end{equation}
and 
\begin{equation}\label{nimm2}
\Ds(\phi \bullet M) \Peq  \phi \bullet (\Ds M)\Peq  (\phi 
1_{(s, \infty)}) \bullet M.
\end{equation}

Next we define and study a \textit{stochastic increment integral} with
respect an increment    local  martingale.  
For   $M\in \mathcal{ILM}(\mcalF_{\bdot})$  set 
\begin{align*}
\mathcal{L}L^1(M)&:= \{\phi: \phi \mbox{ is predictable and } 
\Big(\big(\int_{(-\infty, t]} \phi^2_s \,
\dd[M]_s^\gen\big)^{\tfrac{1}{2}}\Big)_{t\in \R}  \in
\mathcal{LA}_0^1(\mcalF_{\bdot}) \} \\
\mathcal{IL}L^1(M)&:= \{\phi: 
 \phi \in \mathcal{L}L^1(\Ds M) \mbox{ for all } s\in \R\}.
\end{align*}
As an example, if $M\in \mathcal{ILM}^2(\mathcal{F}_{\bdot})$ then a predictable
$\phi$ is in $\mathcal{L}L^1(M)$ resp.\ in  $\mathcal{IL}L^1(M)$ if
(but in general not only if) 
$\int_{(-\infty,t]} \phi_s^2\, \dd \langle M\rangle_s^\gen  <\infty$ for
all $t\in \R$ $P$-a.s.\ resp.\ $\int_{(s,t]} \phi_u^2\, \dd \langle
M\rangle_u^\gen  <\infty$ for all $s<t$ $P$-a.s. If  $M\in
\mathcal{ILM}^2(\mathcal{F}_{\bdot})$ is continuous then 
\begin{align*}
\mathcal{L}L^1(M)&=\{\phi: \phi\mbox{ is predictable and }
\int_{(-\infty, t]}\phi_s^2\, \dd\langle M\rangle^\gen_s <\infty \
P\mbox{-}a.s. \mbox{ for all } t\} \\ 
\mathcal{IL}L^1(M)&=\{\phi: \phi\mbox{ is predictable and }
\int_{(s, t]}\phi_u^2\, \dd\langle M\rangle^\gen_u <\infty \
P\mbox{-}a.s. \mbox{ for all } s<t\}. 
\end{align*} 

Let    $M\in \mathcal{ILM}(\mcalF_{\bdot})$.  The stochastic integral $\phi \bullet (\Ds M)$
 of  $\phi$ in  $\mathcal{IL}L^1(M)$ exists  for all $s\in \R$; in addition, $\{\phi
\bullet (\Ds M)\}_{s\in \R}$ is a consistent family of increment
processes. Indeed,  for $s\leq t\leq u$ we must verify 
\begin{equation*}
(\phi \bullet (\Ds M))_u = (\phi \bullet (\Ds M))_t +
(\phi \bullet (\Dt M))_u, \quad P\mbox{-a.s.}
\end{equation*}
or equivalently 
\begin{equation*}
\Dt(\phi\bullet (\Ds M))_u = (\phi \bullet (\Dt M))_u \quad
P\mbox{-a.s.}, 
\end{equation*} 
which follows from \eqref{kl} and \eqref{nimm2}. Based on this, we
define the \textit{stochastic increment integral of $\phi$ with
  respect to $M$}, to be denoted $\phi \inbull M$, as a \ca\ process
associated with the the family $\{\phi \bullet (\Ds M)\}_{s\in \R}$.
Note that the increment integral $\phi \inbull M$ is uniquely
determined only up to addition of a random variable and it is an
increment local martingale. For $s<t$ and $\phi \in
\mathcal{IL}L^1(M)$ we think of $\phi \inbull M_t - \phi\inbull M_s$
as the integral of $\phi$ with respect to $M$ over the interval
$(s,t]$ and hence use the notation
\begin{equation}
\int_{(s,t]} \phi_u\, \dd M_u:= \phi \inbull M_t - \phi\inbull M_s\
\quad \mbox{for } s<t. 
\end{equation}
When  $\phi\inbull M_{-\infty}:= \lim_{s\to -\infty}
\phi\inbull M_s$ exists $P$-a.s.\ we define the \textit{improper
integral of $\phi$ with respect to $M$ from $-\infty$ to $t$} for
$t\in \R$  as 
\begin{equation}
\int_{(-\infty,t]} \phi_u\, \dd M_u := \phi\inbull M_t -
\phi\inbull M_{-\infty}.
\end{equation} 
Put differently, the improper integral $(\int_{(-\infty,t]} \phi_u\,
\dd M_u)_{t\in \R}$ is, when it exists,   the unique,  up to
$P$-indistinguishability,  increment integral of $\phi$ with respect to
$M$ which is $0$ in  $-\infty$. Moreover, it  is an adapted process. 

The following summarises some fundamental properties. 

\begin{theorem}\label{inc} Let $M\in \mathcal{ILM}(\mcalF_{\bdot})$. 
\begin{enumerate}
\item[(1)] Whenever  $\phi \in \mathcal{IL}L^1(M)$ and $s<t$ we have
  $\Ds (\phi\inbull M)_t = (\phi\bullet (\Ds M))_t$ $P$-a.s.  
\item[(2)]  $\phi\inbull M \in \mathcal{ILM}(\mcalF_{\bdot})$ for  all $\phi \in
    \mathcal{IL}L^1(M)$. 
\item[(3)] If  $\phi, \psi  \in   \mathcal{IL}L^1(M) $  and $a, b \in
  \R$ then  $
(a\phi + b\psi)\inbull M \inl a (\phi \inbull M) + b(\psi  \inbull  M)$. 
\item[(4)] For  $\phi \in  \mathcal{IL}L^1(M)$  we have 
\begin{align}\label{h1}
&\Delta \phi \inbull M_t = \phi_t \Delta M_t, \quad   \mbox{ for } t\in
    \R,\ P\mbox{-a.s.}\\ \label{h2}
&\Dsu [\phi \inbull  M]_t^{\gen}   = \int_{(s,  t]} \phi_u^2 \, \dd
  [M]_s^{\gen}\quad  
    \mbox{for } s \leq t \  P\mbox{-a.s.} 
\end{align}
In particular $[\phi\inbull M]$ exists if and only if $\int_{(-\infty,
  t]}\phi_s^2\, \dd [M]^\gen_s <\infty$ for all $t\in \R$ $P$-a.s.
\item[(5)] If  $\sigma$ a stopping time  and $\phi \in   \mathcal{IL}L^1(M) $
     then  
\begin{equation*}
(\phi \inbull  M)^\sigma \inl (\phi 1_{(-\infty, \sigma]})\inbull  M \inl 
\phi \inbull  (M^\sigma).  
\end{equation*}
\item[(6)] Let  $\phi \in \mathcal{IL}L^1(M)$ and $\psi=(\psi_t)_{t\in
    \R}$ be predictable. Then $\psi\in   \mathcal{IL}L^1(\phi
  \inbull M)$ if and only if $\phi\psi \in \mathcal{IL}L^1(M)$,  and
  in this case $\psi\inbull (\phi \inbull M) \inl (\psi\phi)\inbull
  M$. 
\item[(7)] Let  $\phi \in \mathcal{IL}L^1(M)$. Then $\phi\inbull
  M_{-\infty}:= \lim_{s\to -\infty} \phi\inbull M_s$ exists
  $P$-a.s.\ and $(\int_{(-\infty,t]} \phi_u\, \dd M_u)_{t\in \R}
  \in \mathcal{LM}(\mcalF_{\bdot})$ if and only if $\phi \in
  \mathcal{L}L^1(M)$. 
\end{enumerate}
\end{theorem}

\begin{remark} (a) When  $M$ is continuous it follows from Theorem
  \ref{vbb} that (7) can be simplified to the statement that  $\phi\inbull
  M_{-\infty}= \lim_{s\to -\infty} \phi\inbull M_s$ exists $P$-a.s.\  if and
  only if $\phi \in  \mathcal{L}L^1(M)$, and in this case
  $(\int_{(-\infty, t]}\phi_u \, \dd M_u)_{t\in \R}  \in
  \mathcal{LM}(\mcalF_{\bdot})$. 

(b)  Result (7) above gives a necessary and sufficient condition for
the improper integral to exist and be a local martingale; however, 
improper integrals may exist without being a local martingale (but as
noted above they are  always   increment local  martingales). 
 For example, assume $M$
  is purely discontinuous and that  the compensator  $\nu^M$ of the jump
  measure  $\nu^M$   can be decomposed as
$\nu^M(\cdot; \dd t \times \dd x)=
F(\cdot;t,  \dd x) \mu(\dd t) $ where $F(\cdot;t,  \dd x)$ is a symmetric
measure and $\mu(\{t\})=0$ for all   $t\in \R$. 
 Then by  Theorem \ref{dbb} (3), $\phi\inbull
  M_{-\infty}$ exists $P$-a.s.\ if and only if the quadratic
  variation $[\phi \inbull M]$ exists; that is, 
\begin{equation*}
\sum_{s\leq 0} \phi_s^2 (\Delta M_s)^2 <\infty\quad P\mbox{-}a.s. 
\end{equation*}
 
\end{remark}

\begin{proof}
  Property (1) is merely by definition, and (2) is due to the fact
  that $\Ds (\phi \inbull\nobreak M) \Peq \phi \bullet \Ds M$, which is a
  local martingale.

  (3) We must show that $a(\phi \inbull M) + b(\psi \inbull M)$ is
  associated with $\{ (a\phi +\nobreak b \psi)\bullet (\Ds M)\}_{s\in \R}$,
  i.e.\ that $ \Ds\big(a(\phi \inbull M) + b(\psi \inbull M)\big) \Peq
  (a\phi +b \psi)\bullet (\Ds M)$.  However, by definition of the
  stochastic increment integral and linearity of the stochastic
  integral we have
\begin{align*}
a\, \Ds\big(\phi \inbull  M\big) + b\,  \Ds\big(\psi \inbull  M\big)
\Peq  a \big(\phi \bullet (\Ds M)\big)  +
b\big(\psi \bullet (\Ds M)\big) \Peq  (a \phi+  b\psi) \bullet (\Ds M).
\end{align*}

(4) Using that  $\Ds (\phi \inbull M) =
\phi \bullet (\Ds M)$ and  $\Delta \phi \bullet (\Ds M)\Peq  \phi \Delta
(\Ds M)$,  the result in \eqref{h1} follows. 
  By definition, $[\phi \inbull  M]^{\gen}$ is  associated with 
$
\{[\Ds(\phi \inbull M)]\}_{s\in \R} = \{[\phi\bullet (\Ds M)]\}_{s\in \R}
$.
That is, for $s \in \R$ we have, using that  $[M]^{\gen}$ is
associated with $\{[\Ds M]_s\}_{s\in \R}$,  
\begin{align*}
\Dsu [\phi \inbull M]^{\gen}_t &=[\phi \bullet (\Ds M)]_t
 = \int_{(s,t]} \phi_u^2 \, \dd [\Ds M]_u \\
&=\int_{(s,t]} \phi_u^2 \,
  \dd (\Dsu [ M]^{\gen})_u =  \int_{(s,t]} \phi_u^2 \,
  \dd [ M]_u^{\gen} \quad \mbox{for } s\leq t \quad  P\mbox{-a.s.},
\end{align*} 
which  yields \eqref{h2}. The last statement in (4) follows from
Remark \ref{lk} (1).

The proofs of (5) and (6)  are  left to the reader.

(7) Using (4) the result follows immediately from  Theorem \ref{dbb}. 
\end{proof}

Let us turn to the definition of a  stochastic integral $\phi\bullet
M$ of a predictable $\phi$ with respect to an  increment local 
martingale $M$. Thinking of $\phi\bullet M_t$ as an integral from
$-\infty$ to $t$ it seems reasonable to say that $\phi\bullet M$
(defined for a suitable class of predictable processes $\phi$)  is a stochastic
integral with respect to $M$ if the following is satisfied: 
\begin{itemize}
\item[(1)] $\lim_{t\to -\infty} \phi\bullet M_t=0$ $P$-a.s.
\item[(2)] $\phi_t\bullet M_t - \phi\bullet M_s=   \int_{(s,t]} \phi_u
  \, \dd M_u$ $P$-a.s.\ for all $s<t$
\item[(3)] $\phi\bullet M$ is a local martingale.  
\end{itemize}
By definition of $\int_{(s,t]}\phi_u\, \dd M_u$, (2)  implies
that $\phi\bullet M$ must be    an increment integral of $\phi$ with respect to
$M$.  Moreover, since we assume  $\phi\bullet M_{-\infty}=0$,     $\phi\bullet M$  is uniquely determined as
$(\phi\bullet M_t)_{t\in
  \R} \Peq (\int_{(-\infty, t]} \phi_u \, \dd M_u)_{t\in \R}$, i.e.\   the improper integral of $\phi$.  Since we
also insist that $\phi\bullet M$ is a local martingale, Theorem \ref{inc}~(7) shows that $\mathcal{L}L^1(M)$ is the largest possible set on
which $\phi\bullet M$ can be defined.  We summarise these findings as follows.

\begin{theorem} 
Let $M\in \mathcal{ILM}(\mathcal{F}_{\bdot})$. Then there exists a unique
stochastic integral $\phi\bullet M$  defined  for $\phi \in \mathcal{L}L^1(M)$.  This integral is
given by 
\begin{equation}
\phi\bullet M_t = \int_{(-\infty, t]} \phi_u\, \dd M_u \quad \mbox{for
} t\in \R
\end{equation}
and it satisfied the following.   

\begin{enumerate}
\item[(1)]  $\phi\bullet M \in \mathcal{LM}(\mcalF_{\bdot})$ and 
  $\phi\bullet M_{-\infty}=0$  for  $\phi \in \mathcal{L}L^1(M)$.
\item[(2)] The mapping $\phi \mapsto \phi\bullet M$  is, up to
  $P$-indistinguishability, linear in $\phi\in\mathcal{L}L^1(M) $.
\item[(3)] For $\phi \in\mathcal{L} L^1(M)$ we have 
\begin{align*} 
&\Delta \phi \bullet M_t = \phi_t \Delta M_t, \quad   \mbox{ for } t\in
    \R,\ P\mbox{-a.s.}\\
&[\phi \bullet M]_t  = \int_{(-\infty, t]} \phi_s^2 \, \dd
  [M]_s^{\gen}\quad   
    \mbox{for } t\in \R,\  P\mbox{-a.s.} 
\end{align*}
\item[(4)] For $\sigma$ a stopping time,  $s\in \R$ and $\phi \in
  \mathcal{L}L^1(M)$   we  have 
\begin{equation*}
(\phi \bullet M)^\sigma \Peq (\phi 1_{(-\infty, \sigma]})\bullet M
\Peq 
\phi \bullet (M^\sigma) 
\end{equation*}
and $\Ds (\phi \bullet M) \Peq  \phi \bullet (\Ds M)$. 
\end{enumerate}
\end{theorem}

\begin{example} Let $X\in \mathcal{ILM}(\mcalF_{\bdot})$ be continuous and
  assume  there is a positive continuous predictable process
  $\sigma=(\sigma_t)_{t\in \R}$ such that for all $s<t$, $\Dsu 
  [X]_t^\gen 
  =\int_s^t \sigma_u^2\, \dd u$. Set $B= \sigma^{-1} \inbull X$ and
  note that by L\'evy's theorem $B$ is a standard Brownian motion
  indexed by  $\R$, and $X$ is given by $X\inl \sigma\inbull B$. 
\end{example}

\begin{example}
As a last example assume  $B=(B_t)_{t\in \R}$ is a Brownian motion
indexed by   $\R$ and consider the filtration
$\mathcal{F}_{\bdot}^{\mathcal{I}B}$ generated by the increments of $B$ cf.\
Example \ref{qsx}. In this case a predictable $\phi$ is in
$\mathcal{L}L^1(B)$  resp.\ $\mathcal{IL}L^1(B)$ if and only if
$\int_{-\infty}^t \phi_u^2\, \dd u<\infty$ for all $t$ $P$-a.s.\
resp.\ $\int_{s}^t \phi_u^2\, \dd u<\infty$ for all $s< t$ $P$-a.s.
Moreover,  if $M\in \mathcal{ILM}(\mathcal{F}_{\bdot}^{\mathcal{I}B})$ then
there is a $\phi \in \mathcal{IL}L^1(B)$ such that 
\begin{equation}\label{uh}
M\inl \phi\inbull  B
\end{equation}
and if  $M\in \mathcal{LM}(\mathcal{F}_{\bdot}^{\mathcal{I}B})$ then
there is a $\phi \in \mathcal{L}L^1(B)$ such that 
\begin{equation} \label{uh1}
  M\Peq  M_{-\infty} +  \phi\bullet B. 
\end{equation}
That is, we have a \textit{martingale representation result} in the  filtration
$\mathcal{F}_{\bdot}^{\mathcal{I}B}$. To see that this is the case, it
suffices to prove \eqref{uh}. Let $s\in \R$ and set $\mathcal{H}=
\mathcal{F}_s^{\mathcal{I}B}$. Since $\mathcal{F}_t^{\mathcal{I}B}=
\mathcal{H} \vee  \sigma(B_u-B_s : s\leq u\leq t)$ for $t\geq s$ it
follows  from  Jacod and Shiryaev (2003), Theorem III.4.34,  that
there is a  $\phi^s$ in
$\mathcal{L}L^1(^sB)$ such that $\Ds M\Peq \phi^s \bullet (\Ds
B)$. If $u<s$ then by \eqref{kl} and \eqref{nimm2} we have
$\Ds M= \phi^u  \bullet (\Ds B)$; thus, there is a $\phi$ in 
$ \mathcal{IL}L^1( B)$  such that $\Ds M \Peq \phi \bullet
(\Ds B)$ for all $s$  and hence $M \inl \phi\inbull  B$ by definition of the
increment integral.

The above generalises in an obvious way  to the case where instead
of a Brownian motion $B$ we have, say, a L\'evy process $X$ with
integrable centred increments. In this case, we have to add an integral with
respect to $\mu^X-\nu^X$ on the right-hand sides of \eqref{uh} and
\eqref{uh1}. 

\end{example}

\end{document}